\documentclass[preprint,11pt]{elsarticle}
\usepackage{lipsum}
\makeatletter
\def\ps@pprintTitle{
 \let\@oddhead\@empty
 \let\@evenhead\@empty
 \def\@oddfoot{}
 \let\@evenfoot\@oddfoot}
 \renewcommand{\MaketitleBox}{
  \resetTitleCounters
  \def\baselinestretch{1}
  \begin{center}
    \def\baselinestretch{1}
    \Large \@title \par
    \vskip 18pt
    \normalsize\elsauthors \par
    \vskip 10pt
    \footnotesize \itshape \elsaddress \par
  \end{center}
  \vskip 12pt
}

\usepackage{amssymb}
\usepackage{amsthm}
\usepackage{lineno}
\usepackage{amsfonts}
\usepackage{amsmath}
\usepackage{graphicx}
\usepackage{cases}
\usepackage{epstopdf}
\usepackage{caption}

\usepackage{algorithmic}
\usepackage{booktabs}
\biboptions{sort&compress}
\usepackage{graphicx}
\usepackage{subfigure}
\usepackage{subcaption}
\usepackage{tikz,tabstackengine,amsmath}
\usetikzlibrary{automata,positioning, arrows.meta}

\ifpdf
  \DeclareGraphicsExtensions{.eps,.pdf,.png,.jpg}
\else
  \DeclareGraphicsExtensions{.eps}
\fi

\usepackage{amssymb}
\usepackage{mathtools}

\usepackage[mathscr]{eucal}
\usepackage{color}
\usepackage{geometry}
\usepackage{tikz}
\usetikzlibrary{shapes,arrows}
\usetikzlibrary{decorations.pathreplacing,angles,quotes}

\usepackage{todonotes}

\usepackage{dsfont}
\usepackage{cleveref}
\usepackage{verbatim}
\usepackage{graphicx}
\usepackage{subfig}
\usepackage{tikz}
\usepackage{pgfplots}

\usepgfplotslibrary{groupplots}

\newcommand{\mb}{\mathbb}

\newcommand{\mc}{\mathcal}

\newcommand{\mfN}{{\mathfrak{N}}}
\newcommand{\mfF}{{\mathfrak{F}}}

\newcommand{\mfT}{{\mathcal{T} }}

\newtheorem{theorem}{Theorem}[section]

\newtheorem{proposition}[theorem]{Proposition}

\numberwithin{equation}{section}

\begin{document}

\begin{frontmatter}

\title{Ranking Quantilized Mean-Field Games with an Application to Early-Stage Venture Investments \footnote{This work was first presented at Optimization Days held in Montreal, QC, Canada, from May 6th to 8th, 2024.}\,\footnote{Dena Firoozi and Michèle Breton would like to acknowledge the support of the Natural Sciences and Engineering Research Council of Canada
(NSERC), grants RGPIN-2022-05337, DGECR-2022-00468, and RGPIN-2020-05053, respectively.   
}}

\author[inst1]{Rinel Foguen Tchuendom}

\author[inst2]{Dena Firoozi}

\author[inst1]{Michèle Breton}

\affiliation[inst1]{organization={Department of Decision Sciences},
            addressline={HEC Montréal}, 
            city={Montreal},
            state={QC},
            country={Canada}\\(email: rinel.foguent@gmail.com and   michele.breton@hec.ca)}

\affiliation[inst2]{organization={Department of Statistical Sciences},
            addressline={University of Toronto}, 
            city={Toronto},
            state={ON},
            country={Canada\\ (email: dena.firoozi@utoronto.ca)}}

\begin{abstract}
Quantilized mean-field game models involve quantiles of the population's distribution. We study a class of such games with a capacity for ranking games, where the performance of each agent is evaluated based on its terminal state relative to the population's $\alpha$-quantile value, $\alpha \in (0,1)$. This evaluation criterion is designed to select the top $(1-\alpha)\%$ performing agents. We provide two formulations for this competition: a target-based formulation and a threshold-based formulation. In the former and latter formulations, to satisfy the selection condition, each agent aims for its terminal state to be \textit{exactly} equal and \textit{at least} equal to the population's $\alpha$-quantile value, respectively. 

For the target-based formulation, we obtain an analytic solution and demonstrate the $\epsilon$-Nash property for the asymptotic best-response strategies in the $N$-player game. Specifically, the quantilized mean-field consistency condition is expressed as a set of forward-backward ordinary differential equations, characterizing the $\alpha$-quantile value at equilibrium. For the threshold-based formulation, we obtain a semi-explicit solution and numerically solve the resulting quantilized mean-field consistency condition.
 
Subsequently, we propose a new application in the context of early-stage venture investments, where a venture capital firm financially supports a group of start-up companies engaged in a competition over a finite time horizon, with the goal of selecting a percentage of top-ranking ones to receive the next round of funding at the end of the time horizon. We present the results and interpretations of a set of numerical experiments for both formulations discussed in this context, which illustrate that the target-based formulation closely approximates the threshold-based formulation in the scenarios considered.

\end{abstract}


\end{frontmatter}

\section{Introduction}
\label{sec:introduction}

Mean-field games (MFGs) have emerged as a mathematical framework for modeling the behavior of large populations of interacting agents in a variety of fields, ranging from economics and finance to physics, engineering, and social sciences. The core idea of MFGs is to approximate the collective behavior by considering only interactions between a representative agent and the rest of the population, thus providing a bridge between micro-level individual behavior and macro-level outcomes. In MFGs, each agent is weakly coupled with others through the empirical distribution of their states or control inputs. As the number of agents approaches infinity, this distribution converges to what is known as the mean-field distribution. The behavior of agents in such large populations, as well as the resulting equilibrium, can be approximated by the solution of corresponding infinite-population games (see, for example, \citep{huang2006large, huang2007large, lasry2007mean, carmona2018probabilistic, bensoussan2013mean, cardaliaguet2019master}).  The success of this approach is mainly due to the reduction in the computational complexity associated with the large number of agents. 

MFG theory is still developing, and numerous open problems remain. These include the extension of MFG models to handle the incorporation of more complex interaction structures and the development of numerical methods applicable to practical situations.

Moreover, MFGs have found applications in various domains, particularly within financial markets, where they may be used to model a wide array of problems. Specifically, applications include systemic risk \citep{carmona2013mean,Bo2015SystemicRiskInterbanking,chang2022Systemic}, price impact and optimal execution \citep{casgrain_meanfield_2020, FirooziISDG2017, Cardaliaguet2018Meanfieldgame, huang2019mean}, portfolio trading \citep{Cardaliaguet2018Meanfieldgame,Fu-Horst-2021,Fu-Horst-2022}, equilibrium pricing \citep{Firoozi2022MAFI, gomes_mean-field_2018, fujii_mean_2021,firoozi_shrivats_jaimungal_2022_principalagentmeanfield}, and electricity markets \citep{Rene-Aid-2021,alasseur2023}.

A recent development, related to the incorporation of more complex interaction structures,  involves generalizing the classical linear-quadratic Gaussian setting, where the model only includes the mean value of the distribution of state or control across the agent population, to scenarios including $\alpha$-quantiles of the distribution. 
The existing literature using quantile representation of the population in MFGs is nascent and includes \citep{tchuendom2019quantilized,tchuendom2024class,gao2025linear,tembine2017}. We also note that a class of stochastic partial differential equations with coefficients depending on quantiles are studied in \citep{crisan2014conditional}.

In this work, we first address a new class of quantilized MFG models where $\alpha$-quantiles of the population's state distribution are employed to measure the relative performance of participating agents. More specifically, this class of MFGs represents ranking competitions held by a coordinator who aims to select a certain proportion $(1-\alpha)$, $\alpha \in (0,1)$, of top-ranking agents at the end of the competition, without differentiating between the successful ones. Hence, agents whose terminal state is at least equal to the sample $\alpha$-quantile of the population will succeed. We provide two formulations for this competition, based on the behavior of the pool of agents: \emph{target-based} and \emph{threshold-based} formulations. 
In the target-based formulation, agents aim for their terminal state to be \textit{exactly} equal to the required threshold. For this model, we are able to obtain an analytical solution. 
In the threshold-based formulation, agents aim for their terminal state to be \textit{at least} equal to the required threshold. This model is intuitive, but challenging mathematically, and does not admit a closed-form solution. Through numerical experiments, we show that the target-based formulation closely approximates the threshold-based formulation across all considered scenarios. Related literature on this topic includes recent works \citep{ankirchner2024mean,Ankirchner-MOR-2024}, where agents' dynamics are modeled as oscillating Brownian motions with volatility controlled by the agent. This choice of dynamics enables the application of properties of oscillating Brownian motions, which are not applicable to our setup. Another set of works \citep{bayraktar2016rank, Bayraktar2019LargeTournament,erhanyuchong2021ranking} studies competitions where agents receive rewards at the terminal time based on a bounded function. This function depends on the terminal state value and the agent's rank, which is determined by the proportion of agents whose terminal state value is at most equal to that of the agent. Unlike our setup, this reward function is bounded, which leads to differences in the analysis. A related but distinct line of research concerns risk-sensitive MFGs, which have been studied, for instance, in \citep{tembine2013risk, saldi2018discrete,moon2016linear,moon2019risk,Minyi-Wang2025risk,FirooziRen-2024-common-noise,Firoozi-Liu-RS-2025}.

We then propose a new application domain of MFGs in  financial markets, that is, early-stage venture capital investments. These investments refer to the funding provided to startup companies and entrepreneurial ventures at the beginning of their development cycle. This type of investment is risky but often crucial for the growth of new products and ideas. We use the concept of ranking quantilized MFGs presented in this work to address the portfolio selection problem faced by a venture capital firm through a dynamic game. 
Specifically, we examine a scenario in which a venture capital firm financially supports a group of start-up companies engaged in a competition over a finite time horizon, with the goal of selecting a proportion of top-ranking ones to receive the next round of funding at the end of the time horizon. 
We provide illustrations and insights into the behavior of the pool of start-up companies and the evolution of their market values over time. 
While the scientific literature on venture investments is extensive, it is, to the best of our knowledge, predominantly based on qualitative approaches, such as interviews with industry participants. A small subset of quantitative studies applies optimization and game theory methods, typically  focusing on the interactions between a single start-up and the venture investors (see, for instance,  \citep{ewens2022venture, archibald2021investment, elitzur2003multi, lukas2016venture}). For studies involving multiple distinct start-ups, we refer to \citep{gornall2020squaring,kim2014portfolio} for analysis of start-up valuation models and venture portfolio information disclosure, respectively. 

The contributions of this paper are summarized as follows:

\begin{itemize}
    \item  We present an MFG model where interactions among agents occur through the sample $\alpha$-quantiles of their states. This formulation introduces a new class of ranking games. For this model, we establish the $\epsilon$-Nash property absent from \citep{gao2025linear}, which considers a related class of games where the $\alpha$-quantile appears as a weight in the cost functional, but it does not capture ranking-type interactions. Additionally, our proof of the $\epsilon$-Nash property is also applicable to the model in \citep{tchuendom2023LQGqMFGs}, where agents interact through $\alpha$-quantiles of control distribution, and it relaxes the assumption of a uniform bound on square-integrable deviating strategies imposed in that work.

    \item We introduce a class of ranking games, where there is a threshold for success defined by a sample $\alpha$-quantile  value. In such games, there is no distinction between successful agents, who are all considered successful as long as they meet this (endogenous) threshold. Recent works \citep{Ankirchner-MOR-2024,Ankirchner-MFE-2024} study a class of ranking games based on population $\alpha$-quantiles, in which each agent controls the diffusion coefficient of an oscillating Brownian motion. However, the analysis in these papers relies on specific properties of such Brownian motions that do not apply to our setup.

     \item We explore a new application related to early-stage venture investments. To the best of our knowledge, this is the first work to propose a dynamic game model to address the portfolio selection problem of a venture capital firm.  
     \item We provide a numerical scheme and interpret the results of the numerical experiments conducted in the context of the venture investment problem under study.  
\end{itemize}

The remainder of the paper is structured as follows. In \Cref{sec:ranking-game}, we present the ranking quantilized MFG models under study. We then examine two specific formulations: a target-based formulation in \Cref{sec:target-general} and a threshold-based formulation in \Cref{sec:threshold-general}. In \Cref{sec:application}, we propose an application of the quantilized ranking games in the context of early-stage venture  investments and discuss the results of numerical experiments. We conclude the paper in  \Cref{sec:conclusion}.

\section{Ranking Quantilized Mean Field Games}\label{sec:ranking-game}
\subsection{$N$-Player Game Model} \label{general-fin-pop}
We model a competition between $N$ homogeneous agents over a finite time horizon $\mc{T}=[0,T]$. This competition is set by a coordinator who aims to select a certain proportion $(1-\alpha)$, $\alpha \in (0,1)$, of top-ranking agents at the end of the time horizon. The selection criterion is predetermined and announced by the coordinator before the competition begins. Subsequently, at the terminal time $T$, the performance of the agents is evaluated. Those whose terminal states are equal to or greater than the sample $\alpha$-quantile value of the  $N$ agents are selected by the coordinator. Thus, the sample $\alpha$-quantile value of the terminal state of the participating agents at time $T$ serves as a success threshold, determined endogenously based on the collective performance of all agents. This competition can be considered as a ranking game with a success threshold, where no distinction is made between successful agents. The challenge of this problem arises from the stochastic nature of the success threshold, which is not known a priori by participating agents. Such competitions may be relevant in various contexts. In this paper, we consider an application to early-stage venture investments in \Cref{sec:application}.

More precisely, we study a competition that consists of $N$ agents, who individually have an asymptotically negligible impact on the system as $N$ tends to infinity, over the time horizon $\mfT=[0,T],\, T< \infty$. For simplicity, we assume that the agents are homogeneous, that is, they share the same model parameters. The dynamics of agent $i$, $i \in\mathfrak{N}=\{1,2,\cdots,N\}$, are given by 
\begin{equation} 
 dx_t^i = \left( \gamma_t + b u_t^i \right) dt + \sigma dw_t^i, \quad x_0^i = \xi^i,\label{general-dynamics-fin-pop}
\end{equation}
where the state and the control action of agent $i$ at time $t$ are denoted, respectively, by $x^{i}_{t}\in \mathbb{R}$ and $u^{i}_{t} \in \mathbb{R}$. Specifically, $u^i_t$ represents the effort exerted by agent $i$ at time $t$ and the positive coefficient $b\in \mathbb{R}_{>0}$ is interpreted as its efficiency strength. The process $\gamma =\{\gamma_t: t \in \mfT\}$ is exogenous and deterministic. It can be viewed as the support, whether financial or of another form, provided by the coordinator. The uncertainty specific to the environment of agent $i$ is modeled by the idiosyncratic one-dimensional Wiener process $w^i$ and the constant volatility $\sigma \in \mathbb{R}_{>0}$. More precisely,  $\{w^{i}\in \mathbb{R},i\in \mfN\}$ are $N$ independent Wiener processes defined on the filtered probability space $\big( \Omega,\mfF,\{\mathcal{F}_{t}^{[N]}\}_{t\in \mc{T}},\mathbb{P} \big)$, where $\mathcal{F}_{t}^{[N]}:=\sigma (x^{i}_0, w^{i}_{s}, i \in \mfN, s \in [0,t])$. The initial conditions $\{\xi^i\}_{i \in \mfN}$ are independent and identically distributed, following a normal distribution $\xi^i\sim \mc{N}(m_0,\nu^2)$. 

Following the coordinator's announcement of the quantile level $\alpha$, agent $i,\, i \in \mfN$, aims to select the control process $u^i$ that minimizes the cost functional 
\begin{equation}\label{general-cost-fin-pop}
    J_i^{[N]}(u^i, u^{-i}) = \mathbb{E} \left[ \int_0^T \frac{r}{2} (u^i_t)^2 dt +   g\left( x_T^i -  q^{\alpha, [N]}_T  \right)   \right], 
\end{equation}
where $u^{-i} \coloneqq  (u^1,\dots,u^{i-1}, u^{i+1},\dots, u^N)$ and $r\in \mathbb{R}_{>0}$ is a positive constant. The running cost expressed through the integral term indicates that the efforts are costly and finite. Moreover, for a fixed $\alpha \in (0, 1)$, $q^{\alpha, [N]}_T$ represents the sample $\alpha$-quantile\footnote{In this paper, the terms ``sample $\alpha$-quantile values'' and ``empirical $\alpha$-quantile values'' are used interchangeably.} value of the terminal states $\{x^i_T\}_{i=1}^N$ of the $N$ agents involved in the system, defined as 
\begin{equation} \label{general-sample-quantile}
   q^{\alpha, [N]}_T :=  \min_{k \in\mfN} \left\{ x^k_T \Bigg| \frac{1}{N} \sum_{i=1}^N \mathds{1}_{\{x^i_T \leq x^k_T\}} \geq \alpha \right\},   
\end{equation}
where $\mathds{1}_{\{.\}}$ denotes the indicator variable. For instance, if $\alpha = 0.9$, an agent $i$ whose terminal state satisfies $x^i_T > q^{0.9, [N]}_T$ is in the top 10 percent of the population at the terminal time $T$. The function $g(x_T^i -  q^{\alpha, [N]}_T )$, for which we consider two specific choices in this paper, measures the distance between agent $i$'s terminal state $x_T^i$ and the required threshold $q^{\alpha, [N]}_T$. This function is assumed to have been agreed upon by all agents following the coordinator's announcement. Agent $i$ aims to minimize the cost functional \eqref{general-cost-fin-pop} by applying efforts $u^i_t$ over the time interval $\mfT$. 

According to \eqref{general-dynamics-fin-pop}-\eqref{general-sample-quantile}, the agents are coupled with each other through the sample $\alpha$-quantile value $q^{\alpha, [N]}_T$ of their terminal states, or equivalently, through the $\alpha$-quantile value of the empirical distribution of the agents' states at the terminal time $T$. We note that the weight $\tfrac{1}{N}$ assigned to the contribution of each agent to the sample $\alpha$-quantile value of the population implies that each agent has a uniform and asymptotically negligible impact on the system as $N$ approaches infinity.

For specific choices of the terminal cost $g( x_T^i -  q^{\alpha, [N]}_T)$, it is desirable to identify a set of best-response strategies $\{u_i^* \}_{i=1}^N$ that yields a Nash equilibrium for the system \eqref{general-dynamics-fin-pop}-\eqref{general-sample-quantile} such that 
\begin{equation}
J_{i}^{[N]}(u^{i,*},u^{-i,*}) = \inf_{u^i \in \mc{U}^{[N]}}   J_{i}^{[N]}(u^{i},u^{-i,*}),\quad  \forall i\in \mfN,
\end{equation}
where $\mc{U}^{[N]}$ denotes the set of admissible strategies for agent $i$ defined as 
\begin{equation}
    \mathcal{U}^{[N]} = \left\lbrace u^i : \Omega \times \mfT \longrightarrow \mathbb{R} \,\, \bigg\vert  \,\,  \text{$u^i$ is $\mc{F}^{[N]}$-adapted and}\,\, \mathbb{E} \left[ \int_0^T  (u_t^i)^2 dt \right] < \infty \right\rbrace. 
\end{equation}
Given that addressing this problem can be challenging for a large number of agents, requiring each agent to solve a high-dimensional optimal control problem and observe the states of others, which is not feasible in practice, one usually seeks an approximate equilibrium through the following steps: 
\begin{itemize}
    \item[(i)] Analyzing the limiting problem where the number $N$ of agents tends to infinity and identifying an equilibrium strategy for a representative agent. This step consists of addressing a one-dimensional stochastic control problem that requires only self-state observations and a quantilized mean-field consistency condition. 
    \item[(ii)] Showing that the set of limiting strategies forms an $\epsilon$-Nash equilibrium for the $N$-player game model described by \eqref{general-dynamics-fin-pop}-\eqref{general-sample-quantile}. 
\end{itemize}
We now present the limiting model and describe the main steps for solving this problem.
\subsection{Limiting Model}\label{sec:general-inf-pop}
In this section, we present the limiting problem where the number of agents, $N$, tends to infinity. This problem involves the limiting distribution $\mu$
 and the limiting ${\alpha}$-quantile $\bar{q}^{\alpha}$ instead of the empirical distribution and sample $\alpha$-quantile value ${q}^{\alpha, [N]}$, respectively. 
 
Let $\mathcal{P}_2 (\mathbb{R})$ be the set of probability laws supported on $\mathbb{R}$ with finite second moment. For any $\alpha \in (0,1)$ and $\mu \in \mathcal{P}_2 (\mathbb{R})$, we introduce the $\alpha$-quantile value function
\begin{equation}
\begin{aligned}
     Q  &: \ (0,1) \times \mathcal{P}_2 (\mathbb{R}) \ \longrightarrow \mathbb{R} \\
     &: \ (\alpha, \mu) \mapsto Q(\alpha, \mu) := \inf \left\{ l \in \mathbb{R} \ \vert \ \mu \left( (- \infty, l] \right) \geq \alpha   \right\},
\end{aligned}
\end{equation}
where $Q(\alpha, \mu)$ corresponds to the generalized inverse of the cumulative distribution function associated with the probability law $\mu \in \mathcal{P}_2 (\mathbb{R})$. By definition, for any random variable following the probability law $\mu$, the probability that it takes values greater than or equal to $Q(\alpha, \mu)$ is less than or equal to $1 - \alpha$. 
 
We now consider the  limiting mean-field game model of a representative agent given by 
 \begin{gather}
        dx_t = \left( \gamma_t + b u_t \right) dt + \sigma dw_t,  \label{gmfg1}\\
        J(u)  = \mathbb{E} \left[ \int_0^T \frac{r}{2} (u_t)^2 dt +  g \left( x_T -  \bar{q}^\alpha_T   \right)   \right],  \label{gmfg2}
    \end{gather} 
  with $x_0 = \xi \sim \mc{N}(m_0,\nu^2)$, which involves the limiting $\alpha$-quantile value $\bar{q}^\alpha$. In the above equations, the index $i$ is dropped since all agents are homogeneous.

Additionally, for a representative agent, the space of admissible control processes is defined as 
\begin{equation}
    \mathcal{U} = \left\lbrace u : \Omega \times \mfT \longrightarrow \mathbb{R} \,\, \bigg\vert  \,\,  \text{$u$ is $\mc{F}$-adapted and}\,\, \mathbb{E} \left[ \int_0^T  (u_t)^2 dt \right] < \infty \right\rbrace, 
\end{equation}
where the filtration  $\mathcal{F} = \{\mathcal{F}_{t}: t\in \mfT\}$ is defined such that
$\mathcal{F}_{t}:=\sigma (x_0, w_{s}, s \in \mfT).$

We aim to find the best-response strategy of a representative agent in the limiting case, also referred to as the (MFG) equilibrium strategy of the agent. This strategy is the best strategy of the representative agent in response to the aggregate behavior (mean-field effect) of agents in the limiting case where the number of agents, $N$, tends to infinity. The set of these best response strategies forms a Nash equilibrium for the limiting model with an infinite number of agents.    

For a fixed $\alpha \in (0,1)$, finding the best-response strategy for a representative agent in the limiting model, \eqref{gmfg1}-\eqref{gmfg2}, involves identifying a pair of real-valued processes $ \left( \bar{q}^\alpha, u^* \right)$, where $u^* \in \mathcal{U}$,
through the two steps detailed below.
 
\textit{(i) Stochastic Control Problem}: We fix the $\alpha$-quantile value process at $q^\alpha = \{q^{\alpha}_t: t \in \mfT\}$, which is assumed to be known. We then solve the resulting stochastic control problem for a representative agent given by 
\begin{gather}
        dx_t = \left( \gamma_t + b u_t \right) dt + \sigma dw_t,  \label{gmfg1-fixed}\\
        J(u)  = \mathbb{E} \left[ \int_0^T \frac{r}{2} (u_t)^2 dt +  g\left( x_T -  {q}^\alpha_T   \right)   \right], \label{gmfg2-fixed}
    \end{gather} 
with $x_0 = \xi \sim \mc{N}(m_0,\nu^2)$. This problem involves finding the  optimal pair $(x^*,u^*)=\{(x^*_t,u^*_t): \ t \in \mfT \}$ such that 
\begin{gather}\label{stoch-cntrl-limit-general}
        u^* = \operatorname*{argmin}_{u \in \mathcal{U}} \mathbb{E} \bigg[ \int_0^T \frac{r}{2} (u_t)^2 dt + g\Big( x_T -  {q}^{\alpha}_T\Big)    \bigg],\\
       d{x_t^*}= ({\gamma_t} + b{u_t^*})dt+\sigma {dw_t}. 
    \end{gather}
    Since the limiting $\alpha$-quantile value is a deterministic quantity and is assumed to be fixed in this step, it is straightforward to solve this problem and to obtain the optimal pair $(x^*, u^*)$. 
    
 The idea of fixing the $\alpha$-quantile value in this step stems from the observation that in the limiting case, where there is an infinite number of agents, the $\alpha$-quantile value of the population distribution remains unchanged if one asymptotically negligible agent unilaterally deviates from the set of equilibrium strategies.

 \textit{(ii) Quantilized Mean-Field Consistency Condition}: We equate the resulting $\alpha$-quantile values $Q(\alpha, \mathcal{L}(x^*_t)),\,t\in\mc{T}$, when the obtained optimal strategy $u^*$ is applied with the assumed $\alpha$-quantile values $\{q^\alpha_t: t\in \mc{T}\}$ used to obtain this strategy in step \textit{(i)} as in 
            \begin{equation}
            q^\alpha_t = Q(\alpha, \mathcal{L}(x^*_t)),\quad t\in \mc{T}, \label{MF-consistency-general}  
            \end{equation}  
where $\mathcal{L}(x^*_T)$ denotes the law of the optimal state. If the above equation admits a fixed point, this fixed point, denoted by $\bar{q}^\alpha$, characterizes the $\alpha$-quantile value at equilibrium. 

Finally, we discuss the choice of the terminal cost function $g(.)$.
\subsection{Terminal Cost}
Regarding the terminal cost in \eqref{general-cost-fin-pop}, various interesting  choices for the function $g(.)$ may be considered. However, even simple choices can render the problem complex and pose significant mathematical challenges due to the presence of $\alpha$-quantiles in the model. We provide two formulations of the competition among agents following the announcement of the coordinator: a target-based formulation presented in \Cref{sec:target-general} and a threshold-based formulation presented in \Cref{sec:threshold-general}. 

For the target-based formulation, we establish the existence and uniqueness of a solution to the limiting problem, characterize it explicitly, and show its $\epsilon$-Nash property. For the threshold-based formulation, we characterize a semi-explicit solution to the resulting limiting stochastic control problem. However, establishing the existence of a unique solution to the resulting quantilized mean-field consistency condition and the $\epsilon$-Nash property remains an open mathematical question. For this formulation, we proceed numerically.

\section{Target-Based Formulation}\label{sec:target-general}
In the target-based formulation, it is assumed that, following the coordinator's announcement, each agent aims for its terminal state to be \textit{exactly} equal to the sample $\alpha$-quantile of the terminal states of the  participating agents. Hence, the sample $\alpha$-quantile at the terminal time $T$ acts as the target that the agents aim to achieve, where this target is determined by the collective performance of the participating agents. 
Clearly, agents that meet the target at the terminal time $T$ fulfill the selection condition set by the coordinator. 

\subsection{$N$-Player Game Model} \label{lq-fin-pop}
In the target-based scenario the dynamics and the cost functional of agent $i,\, i\in \mathfrak{N}$, are, respectively, given by \eqref{general-dynamics-fin-pop} and \eqref{general-cost-fin-pop}-\eqref{general-sample-quantile}, where
\begin{equation}\label{lq-cost-fin-pop}
    g(x_T^i -  q^{\alpha, [N]}_T):=  \frac{\lambda}{2} \left( x_T^i -  q^{\alpha, [N]}_T  \right)^2, 
\end{equation}
 and $\lambda \in \mb{R}_{>0}$ is a positive constant. The interpretations of other parameters and processes involved are the same as those of the general formulation given by \eqref{general-dynamics-fin-pop}-\eqref{general-sample-quantile}. As it can be seen, in the target-based formulation the terminal cost of agent $i$ is quadratic, which models the fact that the agent aims for a specific target, that is $q^{\alpha, [N]}_T$. If at the terminal time $T$, the state of agent $i$ is smaller or greater than the target, it incurs a cost. However, if the terminal state of the agent is equal to the target, no cost is incurred. Agent $i$ chooses its strategy $u^i$ in order to minimize the cost functional defined by  \eqref{general-cost-fin-pop}-\eqref{general-sample-quantile} and \eqref{lq-cost-fin-pop}, subject to the dynamics \eqref{general-dynamics-fin-pop}.
 
We first address the limiting problem where the number of agents $N$ tends to infinity and identify an equilibrium strategy for a representative agent in \Cref{sec:lq-inf-pop}. We then show that the set of limiting strategies forms an $\epsilon$-Nash equilibrium for the $N$-player game model described by \eqref{general-dynamics-fin-pop}-\eqref{general-sample-quantile} and \eqref{lq-cost-fin-pop}, as detailed in \Cref{sec:E-Nash}. 

\subsection{Limiting Model}\label{sec:lq-inf-pop}
In this section, we address the limiting problem where the number of agents, $N$, tends to infinity. The limiting mean-field game model of a representative agent is given by \eqref{gmfg1}-\eqref{gmfg2}, where
 \begin{gather}
        g(x_T -  \bar{q}^\alpha_T)= \frac{\lambda}{2} \left( x_T -  \bar{q}^\alpha_T   \right)^2   .  \label{qmfg2}
    \end{gather} 
The following theorem characterizes the equilibrium strategy of the representative agent by solving the corresponding stochastic control problem and the quantilized mean-field consistency condition described in \Cref{sec:general-inf-pop}.

\begin{theorem}
\label{thm-lgq}
For a fixed quantile level $\alpha \in (0,1)$, there exists a solution pair of real-valued processes $\left( \bar{q}, u^* \right)$ with $u^* \in \mathcal{U}$ to the limiting quantilized MFG problem, described by \eqref{gmfg1}-\eqref{gmfg2} and \eqref{qmfg2}, if and only if there exists a solution $\{ \eta_t, \pi_t, v_t, \phi_t^\alpha, \bar{q}^\alpha_t : \ t \in \mfT \}$ to the set of forward-backward ODEs (FBODEs) given by
\begin{align}
    \frac{d \eta_t}{dt} &= \frac{b^2}{r} \eta^2_t,  &\eta_T &= \lambda, \label{fbode1} \allowdisplaybreaks\\
    \frac{d \pi_t}{dt} &= \frac{b^2}{r} \pi^2_t + 2 \frac{b^2}{r} \eta_t \pi_t,  &\pi_T &= - \lambda, \label{fbode2}  \allowdisplaybreaks\\
    \frac{dv_t}{dt} &= \sigma^2 - 2 \frac{b^2}{r} \eta_t v_t, &v_0 &= \nu^2, \label{fbode3} \allowdisplaybreaks \\
    \frac{d \phi^\alpha_t}{dt} &=  - \frac{\sigma^2}{2} \frac{\mathcal{X}_\alpha}{\sqrt{v_t}} \pi_t, &\phi^\alpha_T &= 0, \label{fbode4} \allowdisplaybreaks \\
    \frac{d \bar{q}^\alpha_t}{dt} &=   \gamma_t - \frac{b^2}{r} \phi^\alpha_t + \frac{\sigma^2}{2} \frac{\mathcal{X}_\alpha}{\sqrt{v_t}}, &q^\alpha_0 &= m_0 + \nu \mathcal{X}_\alpha,   \label{fbode5} 
\end{align}
where $\mathcal{X}_\alpha = Q \left( \alpha, \mathcal{N}(0,1) \right)$ is the $\alpha$-quantile value of the standard normal distribution. 

Moreover, for a representative agent, the best-response strategy at the MFG equilibrium, $\{ u^*_t : \ t \in \mfT\}$, is given by
\begin{equation}
    u^*_t = - \frac{b}{r} \left( \eta_t x^*_t + \pi_t \bar{q}^\alpha_t + \phi^\alpha_t \right),\label{lq-cntrl-inf-pop}
\end{equation}
and the corresponding state process, $\{ x^*_t: \ t \in \mfT\}$, satisfies
\begin{align}\label{lq-state-inf-pop}
    d x^*_t &= \left( \gamma_t - \frac{b^2}{r} \eta_t x^*_t - \frac{b^2}{r} \pi_t \bar{q}^\alpha_t - \frac{b^2}{r} \phi^\alpha_t \right) dt + \sigma dw_t, \quad x^*_0 = \xi.
\end{align}
    
\end{theorem}

\begin{proof}
For a fixed quantile level $\alpha \in (0,1)$, we characterize the solution to the limiting MFG problem, described by \eqref{gmfg1}-\eqref{gmfg2} and \eqref{qmfg2}, by following the steps \textit{(i)-(ii)} as described at the beginning of \Cref{sec:lq-inf-pop}. We first address the stochastic control problem with the $\alpha$-quantile value fixed at $q^\alpha_T$. We then characterize the consistency condition \eqref{MF-consistency-general} via a set of FBODEs that includes the ODE that the limiting $\alpha$-quantile value $\{\bar{q}^\alpha_t: t \in \mfT\}$ satisfies at MFG equilibrium. The resulting system of FBODEs is then simplified to the one presented in the theorem.

\textit{(i) Stochastic Control Problem}: We fix the $\alpha$-quantile value at the terminal time $T$ at $q^\alpha_T$. We then introduce the Hamiltonian  
\begin{align}
    H(u, y) = \frac{r}{2} u^2 + y \left( s + b u \right). 
\end{align}
It is well known from the stochastic maximum principle that the solvability of the stochastic control problem, described by \eqref{gmfg1-fixed}-\eqref{gmfg2-fixed} with $g(x_T -  {q}^\alpha_T)= \frac{\lambda}{2} \left( x_T -  {q}^\alpha_T   \right)^2 $, is equivalent to the solvability of the set of forward-backward stochastic differential equations (FBSDEs) given by
\begin{align}
    d x^*_t &= \left( \gamma_t - \frac{b^2}{r} y_t \right) dt + \sigma dw_t, \qquad\quad x^*_0 = \xi,  \\
    d y_t &= z_t dw_t, \qquad\qquad\qquad\qquad\qquad\,\,\, y_T = \lambda (x^*_T - q^\alpha_T), \label{adjoint-BSDE1}
\end{align}
where the optimal control at $t,\, t\in \mfT,$ is characterized as
\begin{align}
    u^*_t = - \frac{b}{r} y_t. 
\end{align}
To solve the above set of FBSDEs, it is standard procedure to use an ansatz for the adjoint process $y=\{y_t: t\in \mfT\}$ given by
\begin{equation}
    y_t = \eta_t x^*_t + \theta^\alpha_t,
\end{equation}
with the terminal conditions $ \eta_T = \lambda $ and $\theta^\alpha_T = - \lambda q^\alpha_T$ obtained from \eqref{adjoint-BSDE1}. 

We aim to characterize the processes $\eta$ and $\theta^\alpha$. For this purpose, we use Itô's lemma to derive the SDE that the ansatz satisfies. This leads to
\begin{align}
    d y_t &= x^*_t \frac{d \eta_t}{dt} dt + \eta_t d x^*_t + \frac{d \theta^\alpha_t}{dt} dt, \notag\\
    &= x^*_t \frac{d \eta_t}{dt} dt + \frac{d \theta^\alpha_t}{dt} dt + \left( \eta_t \gamma_t - \frac{b^2}{r} \eta^2_t x^*_t - \frac{b^2}{r} \theta^\alpha_t  \eta_t  \right) dt + \sigma \eta_t dw_t, \notag\\
    &= \left( \frac{d \eta_t}{dt} - \frac{b^2}{r} \eta^2_t  \right) x^*_t dt + \left( \frac{d \theta^\alpha_t}{dt} + \eta_t \gamma_t - \frac{b^2}{r} \theta^\alpha_t \eta_t \right) dt + \sigma \eta_t dw_t.\label{adjoint-BSDE2}
\end{align}
In order for the two backward stochastic differential equations (BSDEs) satisfied by the adjoint process $y$, specifically \eqref{adjoint-BSDE1} and \eqref{adjoint-BSDE2}, to match, the following equations must hold for all $t \in \mfT$:
\begin{align}
z_t &= \sigma \eta_t,  \\
\frac{d \eta_t}{dt} &= \frac{b^2}{r} \eta^2_t, \quad\qquad\qquad\qquad \eta_T = \lambda,  \\
\frac{d \theta^\alpha_t}{dt} & = - \eta_t \gamma_t + \frac{b^2}{r} \eta_t \theta^\alpha_t,   \quad\quad\,\,\, \theta^\alpha_T = -\lambda q^\alpha_T,
\end{align}
where the last two backward ordinary differential equations (BODEs) characterzie the processes $\eta$ and $\theta^\alpha$, respectively. 

It follows that the optimal control at time $ t,\, t\in \mfT$, is given by  
\begin{equation}
u^*_t = - \frac{b}{r} y_t = - \frac{b}{r} \eta_t x^*_t - \frac{b}{r} \theta^\alpha_t,
\end{equation}
where
\begin{equation}
d x^*_t = \left( \gamma_t - \frac{b^2}{r} \eta_t x^*_t - \frac{b^2}{r} \theta^\alpha_t  \right) dt + \sigma dw_t, \quad x^*_0 = \xi.\label{optimal-state-eq}
\end{equation}

\textit{(ii) Quantilized Mean Field Consistency Condition}: From \eqref{optimal-state-eq}, we observe that for all $t,\, t\in \mfT$, the law of the optimal state, denoted by $\mathcal{L}(x^*_t)$, is Gaussian. Hence the $\alpha$-quantile value $q^\alpha_t$ can be expressed in terms of the mean and the variance of the optimal state at time $t$ as in
\begin{equation}
  q^\alpha_t = Q(\alpha, \mathcal{L}(x^*_t )) = \mathbb{E} \left[ x^*_t \right] + \mathcal{X}_\alpha \sqrt{ \mathbb{V} \left[ x^*_t \right]}\,, \label{consistency-mean-variance}  
\end{equation}
where $\mathcal{X}_\alpha$ represents the $\alpha$-quantile value of the standard normal distribution. 

To obtain the temporal evolution equations for the mean and variance, we introduce the notations
\begin{equation}
    m_t = \mathbb{E} \left[ x^*_t \right], \quad v_t = \mathbb{V} \left[ x^*_t \right].
\end{equation}

For the mean process, $m$, we have 
\begin{align}
     \frac{d m_t}{dt} := \frac{d \mathbb{E} \left[ x^*_t \right] }{dt} &=  \left( \gamma_t - \frac{b^2}{r} \eta_t \mathbb{E} \left[ x^*_t \right] - \frac{b^2}{r} \theta^\alpha_t \right),  \notag\\
     &=   \left( \gamma_t - \frac{b^2}{r} \eta_t m_t - \frac{b^2}{r} \theta^\alpha_t \right),\label{mean-eq}
\end{align}
with  $\mathbb{E} \left[ x^*_0 \right] = m_0$. For the variance process, $v$, we use Itô's lemma to get 
\begin{align}
    d \left( x^*_t - m_t \right)^2 &=  \left[ \sigma^2 - 2 \frac{b^2}{r} \eta_t  \left( x^*_t - m_t \right)^2 \right] dt + 2 \sigma \left( x^*_t - m_t \right) dw_t, 
\end{align}
and then take the expectation of the solution, which satisfies 
\begin{align}
    \frac{d v_t}{dt} := \frac{d \mathbb{E}\left[ \left( x^*_t - m_t \right)^2 \right] }{dt} &= \sigma^2 - 2 \frac{b^2}{r} \eta_t \mathbb{E}\left[  \left( x^*_t - m_t \right)^2 \right]:= \sigma^2 - 2 \frac{b^2}{r} \eta_t v_t,\label{variance-eq}
\end{align}
with $v_0 = \nu^2$. From \eqref{consistency-mean-variance}-\eqref{variance-eq}, the $\alpha$-quantile value at equilibrium, $\bar{q}^\alpha_t$, is the solution to the ODE 
\begin{align}
    \frac{d \bar{q}^\alpha_t}{dt} &= \frac{d m_t}{dt} + \frac{\mathcal{X}_\alpha}{2 \sqrt{v_t}} \frac{d v_t}{dt}, \\
    &= \left( \gamma_t - \frac{b^2}{r} \eta_t m_t - \frac{b^2}{r} \theta^\alpha_t  \right) +  \frac{\mathcal{X}_\alpha}{2 \sqrt{v_t}} \left( \sigma^2 - 2 \frac{b^2}{r} \eta_t v_t \right), \\
    &= \gamma_t - \frac{b^2}{r} \eta_t \left( m_t + \mathcal{X}_\alpha \sqrt{v_t} \right) + \frac{\sigma^2 \mathcal{X}_\alpha}{2 \sqrt{v_t}} - \frac{b^2}{r} \theta^\alpha_t, \\
    &= \gamma_t - \frac{b^2}{r} \eta_t \bar{q}^\alpha_t + \frac{\sigma^2 \mathcal{X}_\alpha}{2 \sqrt{v_t}} - \frac{b^2}{r} \theta^\alpha_t, 
\end{align}
with $\bar{q}^\alpha_0 = m_0 + \mathcal{X}_\alpha \nu$.

 Putting together the previous two steps, we conclude that the solvability of the limiting MFG problem, described by \eqref{gmfg1}-\eqref{gmfg2} and \eqref{qmfg2}, is equivalent to the solvability of the set of FBODEs given by
    \begin{align}
        \frac{d \eta_t}{dt} &= \frac{b^2}{r} \eta^2_t,  &\eta_T &= \lambda,\label{CE1} \\
        \frac{dv_t}{dt} &= \sigma^2 - 2 \frac{b^2}{r} \eta_t v_t, &v_0 &= \nu^2, \label{CE2}\\
        \frac{d \bar{q}^\alpha_t}{dt} &= \gamma_t - \frac{b^2}{r} \eta_t \bar{q}^\alpha_t - \frac{b^2}{r} \theta^\alpha_t + \frac{\sigma^2}{2} \frac{\mathcal{X}_\alpha} {\sqrt{v_t}},  &\bar{q}^\alpha_0 &= m_0 + \nu \mathcal{X}_\alpha, \label{CE3}\\
        \frac{d \theta^\alpha_t}{dt} &= \frac{b^2}{r} \eta_t \theta^\alpha_t - \gamma_t \eta_t, &\theta^\alpha_T &=- \lambda \bar{q}^\alpha_T. \label{CE4}
    \end{align}

\textit{Simplification of Mean Field Consistency FBODEs:} From \eqref{CE1}-\eqref{CE4}, we observe that while the first two ODEs can be solved sequentially, the latter two are coupled and must be solved simultaneously. To obtain a fully decoupled system of FBODEs, we introduce an ansatz for the process $\theta^\alpha=\{\theta^\alpha_t: t\in \mfT\}$ given by  
\begin{equation}\label{ansatz2}
    \theta^\alpha_t = \pi_t \bar{q}^\alpha_t + \phi^\alpha_t, 
\end{equation}
with $\pi_T = - \lambda$ and $\phi^\alpha_T = 0$, which follow from the terminal condition of \eqref{CE4}.

We differentiate the ansatz given by \eqref{ansatz2} to get
\begin{align}
    \frac{d \theta^\alpha_t}{dt} &= \bar{q}^\alpha_t \frac{d \pi_t}{dt} + \pi_t \frac{d \bar{q}^\alpha_t}{dt} + \frac{d \phi^\alpha_t}{dt}, \notag\\
    &= \bar{q}^\alpha_t \left( \frac{d \pi_t}{dt} - \frac{b^2}{r} \pi^2_t - \frac{b^2}{r} \eta_t \pi_t \right) + \left( \frac{d \phi^\alpha_t}{dt} + \gamma_t \pi_t - \frac{b^2}{r} \pi_t \phi_t^\alpha + \frac{\sigma^2}{2} \frac{\mathcal{X}_\alpha}{\sqrt{v_t}} \pi_t \right).\label{CE4-v2}
\end{align}

Subsequently, we substitute the ansatz in the ODE given by \eqref{CE4} to obtain 
\begin{align}
    \frac{d \theta^\alpha_t}{dt} = \bar{q}^\alpha_t \left( \frac{b^2}{r} \eta_t \pi_t \right) + \left( \frac{b^2}{r} \eta_t \phi_t^\alpha - \gamma_t \eta_t \right).
\end{align}

Finally, we match the corresponding terms in the above ODE with those in \eqref{CE4-v2} to derive the ODEs satisfied by $\pi$ and $\phi^\alpha$, respectively, given by
\begin{align}
    \frac{d \pi_t}{dt} &= \frac{b^2}{r} \pi^2_t + 2 \frac{b^2}{r} \eta_t \pi_t, &\pi_T &= - \lambda, \label{CE5}\\
    \frac{d \phi^\alpha_t}{dt} &=  \left( \frac{b^2}{r} \phi^\alpha_t  - \gamma_t \right) (\eta_t + \pi_t) - \frac{\sigma^2}{2} \frac{\mathcal{X}_\alpha}{\sqrt{v_t}} \pi_t, &\phi_T^\alpha &= 0.
\end{align}

Moreover, substituting the ansatz \eqref{ansatz2} in the ODE \eqref{CE3} satisfied by the $\alpha$-quantile value results in 
\begin{align}
    \frac{d \bar{q}^\alpha_t}{dt} &= - \frac{b^2}{r} (\eta_t + \pi_t) \bar{q}^\alpha_t + \left( \gamma_t - \frac{b^2}{r} \phi^\alpha_t + \frac{\sigma^2}{2} \frac{\mathcal{X}_\alpha}{\sqrt{v_t}} \right). 
\end{align}

Furthermore, from \eqref{CE1} and \eqref{CE5}, we observe that
\begin{align}
    \frac{d (\eta_t + \pi_t)}{dt} &= \frac{b^2}{r} \eta^2_t + \frac{b^2}{r} \pi^2_t + 2 \frac{b^2}{r} \eta_t \pi_t = \frac{b^2}{r} (\eta_t + \pi_t)^2, \qquad 
\end{align}
with $\eta_T + \pi_T = - \lambda + \lambda = 0$, which results in $\eta_t + \pi_t = 0$ for all $t,\, t \in \mfT$. This observation allows for further simplification of the set of mean-field consistency FBODEs, as expressed by \eqref{fbode1}-\eqref{fbode5}. Hence, it follows that the solvability of the limiting quantilized MFG problem under study is equivalent to the solvability of the set of fully decoupled FBODEs given by \eqref{fbode1}-\eqref{fbode5}. 
Using the ansatz given by \eqref{ansatz2} for $\theta^\alpha$, the best-response strategy $\{ u^*_t: t \in \mfT \}$ and the resulting state process $\{ x^*_t: \ t \in \mfT\}$ for a representative agent at equilibrium are, respectively, given by \eqref{lq-cntrl-inf-pop} and \eqref{lq-state-inf-pop}.
The proof is complete. 

\end{proof}

The following proposition establishes the existence of a unique solution to the mean-field consistency equations given by \eqref{fbode1}-\eqref{fbode5}.
\begin{proposition}[Existence and Uniqueness of Solution to Mean-Field Consistency Equations]
    \label{existence}
    There is a unique solution $\{ \eta_t, \pi_t, v_t, \phi_t^\alpha, \bar{q}^\alpha_t:  t \in \mfT  \}$ to the mean-field consistency equations given by \eqref{fbode1}-\eqref{fbode5}.
\end{proposition}

\begin{proof}
We observe that the set of FBODEs given by \eqref{fbode1}-\eqref{fbode5} are fully decoupled. As a result it is enough to show that there is a unique solution $\{ \eta_t:  t \in \mfT  \}$ to (\ref{fbode1}). This is because we can verify that $\pi_t = - \eta_t, \forall t \in \mfT$, and calculate  $\{ v_t, \phi_t^\alpha, \bar{q}^\alpha_t: t \in \mfT  \}$ directly from $ \{ \eta_t, \pi_t: t \in \mfT \} $. The existence and uniqueness of the solution to \eqref{fbode1} is an application of the results from \cite{freiling1996generalized}.
\end{proof}

\subsection{$\epsilon$-Nash Property}\label{sec:E-Nash}
In this section, we show that the best-response strategy given by \eqref{lq-cntrl-inf-pop} in \Cref{thm-lgq}, when employed by the agents in the original $N$-player game, described by \eqref{general-dynamics-fin-pop}-\eqref{general-sample-quantile} and \eqref{lq-cost-fin-pop}, results in an $\epsilon$-Nash equilibrium. 
\begin{theorem}
\label{eps-nash}
Consider the solution $\{ \eta_t, \pi_t, v_t, \phi_t^\alpha, \bar{q}^\alpha_t: t \in \mfT  \}$ to the quantilized mean-field consistency equations given by \eqref{fbode1}-\eqref{fbode5}. Then, the set of strategies $\{ u^{i,*}_t := f(t, x^{i,*}_t, \alpha): t\in\mfT\}_{i=1}^{N}$, where
\begin{equation}\label{fin-player-control}
    f(t,x,\alpha) = - \frac{b}{r} \left( \eta_t x + \pi_t \bar{q}^\alpha_t + \phi^\alpha_t \right), \quad \forall \ (t,x,\alpha) \in \mfT \times \mathbb{R} \times (0,1),
\end{equation}
and
\begin{equation}
    d x^{i,*}_t = \left[ \gamma_t - \frac{b^2}{r} \left( \eta_t x^{i, *}_t + \pi_t \bar{q}^\alpha_t + \phi^\alpha_t \right) \right] dt + \sigma d w^i_t, \quad x^{i,*}_0 = \xi^i \sim \mc{N}(m_0,\nu^2),
\end{equation}
forms an $\epsilon$-Nash equilibrium for the $N$-player game described by \eqref{general-dynamics-fin-pop}-\eqref{general-sample-quantile} and \eqref{lq-cost-fin-pop}. That is, for every $i\in \mfN$, 
\begin{equation*}
J_i^{[N]}(u^{i,*}, u^{-i,*})-{\epsilon^\alpha_N} \leq\inf_{u^i\in \,\mathcal{U}^{[N]}} J_i^{[N]}(u^i,  u^{-i,*}) \leq J_i^{[N]}( u^{i,*},  u^{-i,*}).
\end{equation*}  
Moreover, the Nash approximation error is given by 
\begin{equation}\label{Nash-error}
    \epsilon^\alpha_N = \mathcal{O} \left( \sqrt{\frac{1}{N}} \  \frac{\sqrt{\alpha(1-\alpha)}}{p(T, \bar{q}^\alpha_T)} \right), 
\end{equation}
where $p(T, y)$ is the probability density function of the representative agent's terminal state at the MFG equilibrium. Equivalently, 
\begin{equation}\label{eq:epsilon-N-alpha-explicit}
\epsilon_N^\alpha
=
\mathcal O
\left(
\sqrt{\frac{\alpha(1-\alpha)}{N}}
\sqrt{2\pi v_T}
\exp\left(
\frac{\mathcal X_\alpha^2}{2}
\right)
\right),
\end{equation}
where $\mathcal{X}_\alpha = Q \left( \alpha, \mathcal{N}(0,1) \right)$ and $v_T$ is the variance of the representative agent's terminal state at the MFG equilibrium obtained from \eqref{fbode3}.
\end{theorem}
\begin{proof}
Fix \(i\in\{1,\ldots,N\}\), and let \(\widehat{u}^i\in\mathcal U^{[N]}\) be an arbitrary unilateral deviation, while all other players use the limiting optimal controls \(u^{k,*}\), \(k\neq i\). Let \(\widehat x^i\) denote the corresponding state of player \(i\), namely
\begin{equation}\label{eq:deviating-state}
d\widehat x_t^i=(\gamma_t+b \widehat{u}_t^i)\,dt+\sigma\,dw_t^i,
\qquad
\widehat x_0^i=\xi^i.
\end{equation}
The empirical terminal quantile under this unilateral deviation is given by
\begin{equation}\label{eq:deviating-empirical-quantile}
\widehat q_T^{\alpha,[N],i}
:=
Q\Bigg(
\alpha,
\frac1N
\bigg(
\delta_{\widehat x_T^i}
+
\sum_{k\neq i}\delta_{x_T^{k,*}}
\bigg)
\Bigg),
\end{equation}
whereas, when all players use the limiting optimal controls, it is given by
\begin{equation}\label{eq:equilibrium-empirical-quantile}
q_T^{\alpha,[N],*}
:=
Q\bigg(
\alpha,
\frac1N
\sum_{k=1}^N\delta_{x_T^{k,*}}
\bigg).
\end{equation}
The cost functionals corresponding to the $N$-player game are given by
\begin{equation}\label{eq:finite-N-cost-equilibrium}
J_i^{[N]}(u^{i,*},u^{-i,*})
=
\mathbb E
\left[
\int_0^T \frac r2 |u_t^{i,*}|^2\,dt
+
\frac\lambda2
\left(
x_T^{i,*}
-
q_T^{\alpha,[N],*}
\right)^2
\right],
\end{equation}
and
\begin{equation}\label{eq:finite-N-cost-deviation}
J_i^{[N]}(\widehat{u}^i,u^{-i,*})
=
\mathbb E
\left[
\int_0^T \frac r2 |\widehat{u}_t^i|^2\,dt
+
\frac\lambda2
\left(
\widehat x_T^i
-
\widehat q_T^{\alpha,[N],i}
\right)^2
\right].
\end{equation}
The corresponding limiting cost functional is
\begin{equation}\label{eq:limiting-cost}
J(u)
:=
\mathbb E
\left[
\int_0^T \frac r2 |u_t|^2\,dt
+
\frac\lambda2
\left(
x_T^u-\bar q_T^\alpha
\right)^2
\right],
\end{equation}
where \(d x_t^u=(\gamma_t+b u_t)\,dt+\sigma\,dw_t\) and \(x_0^u=\xi\). In particular, \(J(u^*)\) denotes the value of \eqref{eq:limiting-cost} at the limiting optimal control \(u^*\), and \(J(\widehat u)\) denotes the value of \eqref{eq:limiting-cost} corresponding to the deviating control \(\widehat{u}^i\), viewed as a control in the limiting single-agent problem. By construction,
\begin{equation}\label{eq:mf-optimality}
J(u^*)\leq J(\widehat u).
\end{equation}
It is enough to consider deviations satisfying
\begin{equation}\label{eq:profitable-deviation}
J_i^{[N]}(\widehat{u}^i,u^{-i,*})
\leq
J_i^{[N]}(u^{i,*},u^{-i,*}),
\end{equation}
because otherwise the \(\epsilon\)-Nash inequality is immediate. For such deviations, we have
\[
0
\leq
J_i^{[N]}(u^{i,*},{u}^{-i,*})
-
J_i^{[N]}(\widehat{u}^i,u^{-i,*}).
\]
Adding and subtracting the limiting costs \(J(u^*)\) and \(J(\widehat u)\), we obtain the decomposition
\begin{align}
0
&\leq
J_i^{[N]}(u^{i,*},u^{-i,*})
-
J_i^{[N]}(\widehat{u}^i,u^{-i,*}) \notag\\
&=
\left[
J_i^{[N]}(u^{i,*},u^{-i,*})
-
J(u^*)
\right]
+
\left[
J(u^*)
-
J(\widehat u)
\right] \notag\\
&\quad+
\left[
J(\widehat u)
-
J_i^{[N]}(\widehat{u}^i,u^{-i,*})
\right] \notag\\
&=: I^*+I+\widehat I,
\label{eq:cost-decomposition}
\end{align}
where
\begin{equation}\label{eq:I-star-def}
I^*
:=
J_i^{[N]}(u^{i,*},u^{-i,*})
-
J(u^*),
\end{equation}
\begin{equation}\label{eq:I-def}
I
:=
J(u^*)
-
J(\widehat u),
\end{equation}
\begin{equation}\label{eq:I-hat-def}
\widehat I
:=
J(\widehat u)
-
J_i^{[N]}(\widehat{u}^i,u^{-i,*}).
\end{equation}
By \eqref{eq:mf-optimality}, \(I\leq0\). Hence, from \eqref{eq:cost-decomposition},
\begin{equation}\label{eq:cost-difference-reduced}
0
\leq
J_i^{[N]}(u^{i,*},u^{-i,*})
-
J_i^{[N]}(\widehat{u}^i,u^{-i,*})
\leq
I^*+\widehat I.
\end{equation}
It remains to estimate \(I^*\) and \(\widehat I\).

First, using \eqref{eq:finite-N-cost-equilibrium} and \eqref{eq:limiting-cost}, we have
\begin{align}
I^*
&=
\frac\lambda2
\mathbb E
\left[
\left(
x_T^{i,*}
-
q_T^{\alpha,[N],*}
\right)^2
-
\left(
x_T^{i,*}
-
\bar q_T^\alpha
\right)^2
\right],
\label{eq:I-star-expanded}
\end{align}
where the running cost terms cancel because $u^{i,*}=f(t,x^{i,*},\alpha)$ and $u^*$ are respectively given by \eqref{fin-player-control} and \eqref{lq-cntrl-inf-pop}, and $\mathcal{L}({x}_t^*) = \mathcal{L}({x}^{i,*}_t), \ \forall \ t \in \mfT, \forall i \in \mfN$. Using \((a-b)^2-(a-c)^2=(c-b)(2a-b-c)\), we get
\begin{align}
|I^*|
&\leq
\frac\lambda2
\mathbb E
\left[
\left|
q_T^{\alpha,[N],*}
-
\bar q_T^\alpha
\right|
\left|
2x_T^{i,*}
-
q_T^{\alpha,[N],*}
-
\bar q_T^\alpha
\right|
\right] \notag\\
&\leq
\frac\lambda2
\left(
\mathbb E
\left[
\left|
q_T^{\alpha,[N],*}
-
\bar q_T^\alpha
\right|^2
\right]
\right)^{1/2}
\left(
\mathbb E
\left[
\left|
2x_T^{i,*}
-
q_T^{\alpha,[N],*}
-
\bar q_T^\alpha
\right|^2
\right]
\right)^{1/2}.
\label{eq:I-star-CS}
\end{align}
By Minkowski's inequality,
\begin{align}
\left(
\mathbb E
\left[
\left|
2x_T^{i,*}
-
q_T^{\alpha,[N],*}
-
\bar q_T^\alpha
\right|^2
\right]
\right)^{1/2} 
&\leq
2\|x_T^{i,*}\|_{L^2}
+
\|q_T^{\alpha,[N],*}\|_{L^2}
+
|\bar q_T^\alpha| \notag\\
&\leq
2\|x_T^*\|_{L^2}
+
\left(
\mathbb E
\left[
\left|
q_T^{\alpha,[N],*}
-
\bar q_T^\alpha
\right|^2
\right]
\right)^{1/2}
+
2|\bar q_T^\alpha|.
\label{eq:I-star-second-factor}
\end{align}
Since \(x_T^{1,*},\ldots,x_T^{N,*}\) are i.i.d. with law \(\mathcal N(m_T,v_T)\), the standard \(L^2\)-empirical quantile estimate (see e.g. \cite[Sec. 2.5]{serfling2009approximation},\cite{bahadur1966} and \citep{ghosh1971}) gives
\begin{equation}\label{eq:empirical-quantile-estimate-star}
\left(
\mathbb E
\left[
\left|
q_T^{\alpha,[N],*}
-
\bar q_T^\alpha
\right|^2
\right]
\right)^{1/2}
=
\mathcal O
\left(
\frac{1}{\sqrt N}
\frac{\sqrt{\alpha(1-\alpha)}}{p(T,\bar q_T^\alpha)}
\right),
\end{equation}
where $p(T,y)$ is the probability density function of the limiting terminal state at the MFG equilibrium given by
\begin{equation}
    x^{*}_T = \xi + \int_0^T \left( \gamma_t - \frac{b^2}{r} \left( \eta_t x^{*}_t + \pi_t \bar{q}^\alpha_t + \phi^\alpha_t \right) \right) dt + \sigma  w_T.
\end{equation}
Combining \eqref{eq:I-star-CS}, \eqref{eq:I-star-second-factor}, and \eqref{eq:empirical-quantile-estimate-star}, we obtain
\begin{equation}\label{eq:I-star-bound}
I^*
=
\mathcal O
\left(
\frac{1}{\sqrt N}
\frac{\sqrt{\alpha(1-\alpha)}}{p(T,\bar q_T^\alpha)}
\right).
\end{equation}

We now estimate \(\widehat I\). By \eqref{eq:profitable-deviation} and the nonnegativity of the terminal cost, we have
\[
\frac r2
\mathbb E\int_0^T |\widehat{u}_t^i|^2\,dt
\leq
J_i^{[N]}(\widehat{u}^i,u^{-i,*})
\leq
J_i^{[N]}(u^{i,*},u^{-i,*}).
\]
The right-hand side is uniformly bounded in \(N\), since
$
u_t^{i,*}
=
-\frac br
\left(
\eta_t x_t^{i,*}
+
\pi_t\bar q_t^\alpha
+
\phi_t^\alpha
\right)$, the process \(x^{i,*}\) has uniformly bounded second moments on \(\mathcal{T}\), and \(q_T^{\alpha,[N],*}\) has uniformly bounded second moment. Therefore, there exists a constant \(C>0\), independent of \(N\), \(i\), such that for any profitable deviation, we have
\begin{equation}\label{eq:deviation-control-bound}
\mathbb E\int_0^T |\widehat{u}_t^i|^2\,dt\leq C.
\end{equation}
From \eqref{eq:deviating-state}, using Jensen's inequality, Cauchy--Schwarz, and \eqref{eq:deviation-control-bound}, we obtain
\begin{align}
\mathbb E|\widehat x_T^i|^2
&\leq
C
\left[
\mathbb E|\xi^i|^2
+
\left(\int_0^T|\gamma_t|\,dt\right)^2
+
b^2T\mathbb E\int_0^T |\widehat u_t^i|^2\,dt
+
\sigma^2T
\right] \notag\\
&\leq C.
\label{eq:deviating-state-L2-bound}
\end{align}
Thus \(\|\widehat x_T^i\|_{L^2}\leq C\). Using \eqref{eq:I-hat-def}, \eqref{eq:finite-N-cost-deviation}, \eqref{eq:limiting-cost} and the fact that the running cost terms cancel, we have
\begin{align}
\widehat I
&=
\frac\lambda2
\mathbb E
\left[
\left(
\widehat x_T^i-\bar q_T^\alpha
\right)^2
-
\left(
\widehat x_T^i-\widehat q_T^{\alpha,[N],i}
\right)^2
\right].
\label{eq:I-hat-expanded}
\end{align}
Using again \((a-b)^2-(a-c)^2=(c-b)(2a-b-c)\), we get
\begin{align}
|\widehat I|
&\leq
\frac\lambda2
\mathbb E
\left[
\left|
\widehat q_T^{\alpha,[N],i}
-
\bar q_T^\alpha
\right|
\left|
2\widehat x_T^i
-
\widehat q_T^{\alpha,[N],i}
-
\bar q_T^\alpha
\right|
\right] \notag\\
&\leq
\frac\lambda2
\left(
\mathbb E
\left[
\left|
\widehat q_T^{\alpha,[N],i}
-
\bar q_T^\alpha
\right|^2
\right]
\right)^{1/2}
\left(
\mathbb E
\left[
\left|
2\widehat x_T^i
-
\widehat q_T^{\alpha,[N],i}
-
\bar q_T^\alpha
\right|^2
\right]
\right)^{1/2}.
\label{eq:I-hat-CS}
\end{align}
By Minkowski's inequality and \eqref{eq:deviating-state-L2-bound},
\begin{align}
\left(
\mathbb E
\left[
\left|
2\widehat x_T^i
-
\widehat q_T^{\alpha,[N],i}
-
\bar q_T^\alpha
\right|^2
\right]
\right)^{1/2}
&\leq
2\|\widehat x_T^i\|_{L^2}
+
\|\widehat q_T^{\alpha,[N],i}\|_{L^2}
+
|\bar q_T^\alpha| \notag\\
&\leq
2\|\widehat x_T^i\|_{L^2}
+
\left(
\mathbb E
\left[
\left|
\widehat q_T^{\alpha,[N],i}
-
\bar q_T^\alpha
\right|^2
\right]
\right)^{1/2}
+
2|\bar q_T^\alpha|.
\label{eq:I-hat-second-factor}
\end{align}
It remains to estimate 
$\left\|
\widehat q_T^{\alpha,[N],i}-\bar q_T^\alpha
\right\|_{L^2}$. Let $\widehat F^{[N],i}$ be the empirical CDF of
$\{\widehat x^{i}_T\}\cup\{x^{k,*}_T\}_{k\neq i}$ and $F^{[N-1]}$ that of
$\{x^{k,*}_T\}_{k\neq i}$. Adding a single atom changes the empirical CDF by at most $\frac{1}{N}$. Specifically, define $S(y):=\sum_{k\neq i}\mathds{1}_{\{x^{k,*}_T\le y\}}\in\{0,\dots,N-1\}$. Then,
\[
\widehat F^{[N],i}(y)-F^{[N-1]}(y)
  =\frac{1}{N}\mathds{1}_{\{\widehat x^{i}_T\le y\}}
   +S(y)\Big(\frac1N-\frac1{N-1}\Big)
  =\frac{\mathds{1}_{\{\widehat x^{i}_T\le y\}}}{N}
   -\frac{S(y)}{N(N-1)},
\]
where both terms lie in the interval $[0,\tfrac{1}{N}]$,  and the second inclusion follows from $S(y)\le N-1$. Hence,
$\|\widehat F^{[N],i}-F^{[N-1]}\|_\infty\le \frac{1}{N}$.

For CDFs $F(.)$ and $G(.)$, if $\|F-G\|_\infty\le\delta$, then
$G^{-1}(\beta-\delta)\le F^{-1}(\beta)\le G^{-1}(\beta+\delta)$, by the
Galois property $F(y)\ge\beta\Leftrightarrow y\ge F^{-1}(\beta)$. With
$\delta=2/N$ this gives, $\mathbb{P}$-almost surely\ and uniformly in $\widehat u^{i}$,
\[
  Q\big(\alpha-\tfrac2N,F^{[N-1]}\big)\le\widehat q^{\alpha,[N],i}_T
  \le Q\big(\alpha+\tfrac2N,F^{[N-1]}\big),
\]
which is a bound depending only on the equilibrium sample ($n=N-1$). Subsequently, $\mathbb{P}$-almost surely, we have
\begin{align*}
  \big|\widehat q^{\alpha,[N],i}_T-\bar q^{\alpha}_T\big|
  &\le\max\big(|Q\big(\alpha-\tfrac2N,F^{[N-1]}\big)-\bar q^{\alpha}_T|,\,|Q\big(\alpha+\tfrac2N,F^{[N-1]}\big)-\bar q^{\alpha}_T|\big)\\
  &\le|Q\big(\alpha-\tfrac2N,F^{[N-1]}\big)-\bar q^{\alpha}_T|+|Q\big(\alpha+\tfrac2N,F^{[N-1]}\big)-\bar q^{\alpha}_T|.
\end{align*}
Inserting $\bar q^{\,\alpha\pm2/N}_T$ and
taking $L^2$ norms, yields 
\begin{align}
  \big\|\widehat q^{\alpha,[N],i}_T-\bar q^{\alpha}_T\big\|_{L^2}
  \le& \Big(
   \big\|Q(\alpha-2/N,F^{[N-1]})-\bar q^{\alpha-2/N}_T\big\|_{L^2}
   +\big|\bar q^{\alpha-2/N}_T-\bar q^{\alpha}_T\big|\Big)\notag\\
   &+ \Big(
   \big\|Q(\alpha+2/N,F^{[N-1]})-\bar q^{\alpha+2/N}_T\big\|_{L^2}
   +\big|\bar q^{\alpha+2/N}_T-\bar q^{\alpha}_T\big|\Big)\notag\\
   =&\mathcal O\!\Big(\tfrac{\sqrt{\alpha(1-\alpha)}}
  {\sqrt N\,p(T,\bar q^{\alpha}_T)}\Big)
  +\mathcal O\!\Big(\tfrac{1}{N\,p(T,\bar q^{\alpha}_T)}\Big)
  ,\label{eq:deviating-quantile-bound}
\end{align}
where $L^2$ estimates for i.i.d. sample quantiles ($n=N-1$) together with $\tfrac{d}{d\beta}\bar q^{\beta}_T=1/p(T,\bar q^{\beta}_T)$ were used. The first term on the right-hand side of \eqref{eq:deviating-quantile-bound} accounts for sampling error, while the second arises from the level shift by $\tfrac{2}{N}$ and is of lower order.

Combining \eqref{eq:I-hat-CS}, \eqref{eq:I-hat-second-factor}, \eqref{eq:deviating-state-L2-bound}, and \eqref{eq:deviating-quantile-bound}, we obtain
\begin{equation}\label{eq:I-hat-bound}
\widehat I
=
\mathcal O
\left(
\frac{1}{\sqrt N}
\frac{\sqrt{\alpha(1-\alpha)}}{p(T,\bar q_T^\alpha)}
\right).
\end{equation}

Finally, substituting \eqref{eq:I-star-bound} and \eqref{eq:I-hat-bound} into \eqref{eq:cost-difference-reduced}, and considering any deviating strategy (profitable or non-profitable) $\widehat{u}^i \in \mathcal{U}^{[N]}$, we get
\begin{equation*}\label{eq:epsilon-nash-bound}
J_i^{[N]}(u^{i,*},u^{-i,*})
-
J_i^{[N]}(\widehat{u}^i,u^{-i,*})
\leq
\epsilon_N^\alpha,\qquad \forall \widehat{u}^i \in \mathcal{U}^{[N]},
\end{equation*}
where
\begin{equation}\label{eq:epsilon-N-alpha}
\epsilon_N^\alpha
=
\mathcal O
\left(
\frac{1}{\sqrt N}
\frac{\sqrt{\alpha(1-\alpha)}}{p(T,\bar q_T^\alpha)}
\right).
\end{equation}
Therefore,
$
J_i^{[N]}(u^{i,*},u^{-i,*})
\leq
\inf_{u^i \in \mathcal{U}^{[N]}}J_i^{[N]}(u^i,u^{-i,*})
+
\epsilon_N^\alpha.
$
Since \(x_T^*\sim\mathcal N(m_T,v_T)\) and by \eqref{fbode3}, $v_T = \left(\frac{1}{1 + \frac{b^2}{r}\lambda T}\right)^2 v_0>0$, we have \(\bar q_T^\alpha=m_T+\mathcal X_\alpha\sqrt{v_T}\). Moreover, $
p(T,\bar q_T^\alpha)
=
\frac{1}{\sqrt{2\pi v_T}}
\exp\left(
-\frac{\mathcal X_\alpha^2}{2}
\right).
$ Therefore, \eqref{eq:epsilon-N-alpha} can also be written as \eqref{eq:epsilon-N-alpha-explicit}. This completes the proof.
\end{proof}

\section{Threshold-Based Formulation}\label{sec:threshold-general}
In the threshold-based formulation, it is assumed that following the coordinator's announcement, the agents aim for their terminal state to be \textit{at least} equal to the sample $\alpha$-quantile value of the terminal state of the participating agents. In this scenario, the sample $\alpha$-quantile at the terminal time $T$ may be viewed as the threshold for success by agents, which is determined by the collective performance of all participating agents. The agents that achieve or surpass this threshold at the terminal time $T$ fulfill the selection condition set by the coordinator.

\subsection{$N$-Player Game Model}
In the threshold-based scenario, the dynamics and the cost functional of agent $i,\, i\in \mc{N}$, are, respectively, given by \eqref{general-dynamics-fin-pop} and \eqref{general-cost-fin-pop}-\eqref{general-sample-quantile}, where
\begin{equation}\label{cost-fin-pop-VI}
     g\left(x_T^i -  q^{\alpha, [N]}_T\right):=\frac{\lambda}{2} \left( x_T^i -  q^{\alpha, [N]}_T  \right)^2 \mathds{1}_{\left\{x_T^i <   q^{\alpha, [N]}_T\right\}}   , 
\end{equation}
where $\lambda \in \mb{R}_{>0}$ is a positive constant. According to the terminal cost for agent $i$, if its  state $x^i_T$ is below the sample $\alpha$-quantile value $q^{\alpha, [N]}_T$ at the terminal time $T$, it incurs a cost. This cost is increasing with the distance of the state of the start-up $x^i_T$ to the required threshold $q^{\alpha, [N]}_T$. However, if its state $x^i_T$ is equal to or exceeds the sample $\alpha$-quantile value $q^{\alpha, [N]}_T$ at terminal time $T$, no cost is incurred. Overall, the cost functional given by \eqref{cost-fin-pop-VI} models the fact that agent $i$ aims for its terminal state $x^i_T$ to be in the top $(1-\alpha) \times 100\%$ of the $N$ agents by applying efforts $u^i_t$ throughout the time interval $\mc{T}$. Clearly, the agent's strategy $u^i$ depends on the strategies of other agents through the sample $\alpha$-quantile value $q^{\alpha, [N]}_T$ at time $T$. 

It is desirable to identify an approximate Nash equilibrium for this model. However, this task proves to be mathematically challenging, as detailed in the following subsection. Consequently, we focus on the limiting case as the number of agents \(N\) approaches infinity, combining analytical and numerical methods. Specifically, we characterize an equilibrium strategy in a semi-explicit form and use numerical methods to fully describe it.
\subsection{Limiting Model}\label{sec:VC-nonlinear-sol}   
 We address the limiting problem, where the number of agents tends to infinity, through a similar two-step process as detailed in \Cref{sec:general-inf-pop}.
 
The resulting stochastic control problem for a representative agent is described by \eqref{gmfg1-fixed}-\eqref{gmfg2-fixed}, where
    \begin{gather}
     g\left(x_T -  {q}^{\alpha}_T\right)=\frac{\lambda}{2}\Big(x_T -  {q}^{\alpha}_T\Big)^2 \mathds{1}_{\left\{x_T <   {q}^{\alpha}_T\right\}}   .\label{VI-cost-nonlinear}
    \end{gather}
Similarly to the target-based formulation, this problem involves the limit $\alpha$-quantile value instead of the sample $\alpha$-quantile value. The limit $\alpha$-quantile value is a deterministic quantity and by assuming that it is fixed in this step, the problem is significantly simplified. This is because the success threshold becomes both deterministic and known a priori. Subsequently, a representative agents chooses its best-response strategy to achieve this set threshold by the terminal time $T$.

\begin{proposition} Suppose that the optimal control problem of a representative agent in the limit, as the number of agents $N$ tends to infinity, is given by \eqref{gmfg1-fixed}-\eqref{gmfg2-fixed} and \eqref{VI-cost-nonlinear}, and the $\alpha$-quantile value process $q^\alpha = \{q^{\alpha}_t: t \in \mfT\}$ is fixed. Then, the optimal (best-response) strategy of a representative agent is given by  
        \begin{equation}\label{optimal-strategy-limit}
            u^*_t = -\frac{b}{r} y^*_t, 
        \end{equation} 
        where a complete characterization leads to seeking the triple \((
        x^*_t, y^*_t, z^*_t)\) satisfying 
        \begin{align}
            dx^*_t &= \left[\gamma_t - \tfrac{b^2}{r} y^*_t \right] dt + \sigma dw_t,\qquad x^*_0 \sim \mathcal{N}(m_0, \nu^2),\label{FSDE}\\
            dy^*_t &= z^*_t dw_t, \qquad \qquad \qquad \qquad \,\,\,\, y^*_T = \lambda\left( x_T^* -  q^{\alpha}_T  \right) \mathds{1}_{\left\{x_T^* <   q^{\alpha}_T\right\}}.\label{BSDE}
        \end{align}
        Moreover, the optimal strategy admits the following equivalent representation 
        \begin{align}\label{opt-cntrl-equiv}
            u^*_t =\frac{b}{r} \lambda \left(  q^{\alpha}_T\,\mathbb{E}\left[  \mathds{1}_{\left\{x_T^* <   q^{\alpha}_T\right\}} \Big| \mathcal{F}_t\right] - \mathbb{E}\left[x_T^*   \mathds{1}_{\left\{x_T^* <   q^{\alpha}_T\right\}} \Big| \mathcal{F}_t\right] \right).
        \end{align}
\end{proposition}
\begin{proof}
Using the stochastic maximum principle, the optimal strategy is characterized in a semi-explicit form as given by \eqref{optimal-strategy-limit}, where the adjoint process $\{y^*_t: t \in \mfT\}$ is fully characterized by solving the set of forward-backward SDEs (FBSDEs) given by \eqref{FSDE}-\eqref{BSDE}. From \eqref{BSDE}, we observe that the adjoint process $y^\ast$ is a martingale and hence an equivalent representation of the optimal strategy is given by 
        \begin{equation*}
            u^*_t = -\frac{b}{r} \lambda\,\mathbb{E}\left[\left( x_T^* -  q^{\alpha}_T  \right) \mathds{1}_{\left\{x_T^* <   q^{\alpha}_T\right\}} \Big| \mathcal{F}_t\right],
        \end{equation*}
     where the terms may be rearranged as in \eqref{opt-cntrl-equiv}.
\end{proof}
    The first expectation in \eqref{opt-cntrl-equiv} represents the conditional failure probability of a representative agent given the information available up to time $t$. The second expectation in \eqref{opt-cntrl-equiv} represents the expected terminal state of the agent when it falls below the required threshold, given the information available up to time $t$. Hence, for a fixed quantile level, $\alpha$, and based on the information available at time $t$, if the terminal state of the agent falls below the required threshold, the effort exerted by the start-up at time $t$ increases as the expected deviation of its terminal state from the threshold grows. Conversely, if the agent's probability of failure is zero, its effort reduces to zero.
        
Although we characterized the solution to the corresponding optimal control problem in a semi-explicit form, establishing the existence of a unique fixed point for the associated quantilized consistency condition~\eqref{MF-consistency-general} poses significant theoretical challenges, primarily due to the irregularity of the quantile function with respect to the underlying law~$\mathcal{L}(x_t^*)$. While this remains an open mathematical question, we defer its investigation to future research. In the present work, we address this quantilized mean-field consistency equation through the numerical scheme depicted in \Cref{fig:numerical-scheme}. For similar reasons, the convergence analysis of the numerical algorithm remains also challenging. 

\begin{figure}
\centering
{\scriptsize   
\begin{tikzpicture}[node distance=4cm,auto,scale=0.7, transform shape]
\tikzstyle{block1} = [rectangle, draw, thick,fill=blue!10, text width=6.5em, text centered, rounded corners, minimum height=4em]

\tikzstyle{block2} = [rectangle, draw, thick,fill=blue!10, text width=10em, text centered, rounded corners, minimum height=4em]

\tikzstyle{block} = [rectangle, draw, thick, fill=blue!10, text width=3em, text centered, rounded corners, minimum height=3em]

    \node [block1] (init) {Fix the distribution\\ $\mu$};
    \node [block1, right of=init, node distance=4cm] (identify) {Calculate the $\alpha$-quantile value\\ $q^{\alpha}$};
    \node [block1, right of=identify, node distance=4cm] (evaluate) {Solve the resulting stoch. control prob.\\ $x^\ast, u^\ast$};
    \node [block1, below of=evaluate, node distance=2cm] (decide) {Calculate the distribution of optimal state\\ $\mc{L}(x^\ast)$};
    \node [block1, left of=decide, node distance=4cm] (stop) {Calculate the $\alpha$-quantile value\\
    $Q(\alpha,\mc{L}(x^\ast))$};
    \node [block2, left of=stop, node distance=4cm] (loop) {$|q^{\alpha}- Q(\alpha, \mc{L}(x^\ast))| \overset{?}{<} \delta$};
    \node [block1, below of=loop, node distance=1.8cm] (realstop) {Stop \\ $\bar{q}^\alpha=q^\alpha$};
    
    \path [draw, -latex'] (init) -- (identify);
    \path [draw, -latex'] (identify) -- (evaluate);
    \path [draw, -latex'] (evaluate) -- (decide);
    \path [draw, -latex'] (decide) -- (stop);
    \path [draw, -latex'] (stop) -- (loop);
    \path [draw, -latex'] (loop) -- node {yes} (realstop);
    \path [draw, -latex'] (loop) -- node {no} (init)node[midway, right] {set $\mu=\mc{L}(x^\ast)$};
\end{tikzpicture}}
\caption{Numerical scheme for solving the limiting threshold-based problem as described by \eqref{gmfg1-fixed}-\eqref{gmfg2-fixed} and \eqref{VI-cost-nonlinear}.}
\label{fig:numerical-scheme}
\end{figure}
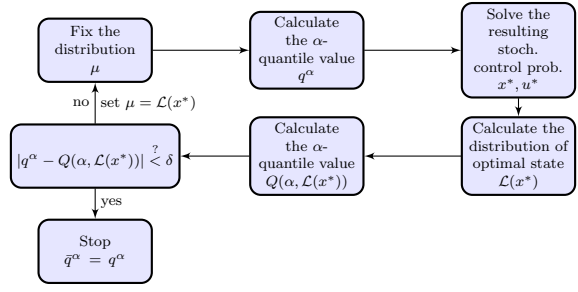

\section{Application to Early-Stage Venture Investments}\label{sec:application}
Early-stage venture investments refer to the funding provided to start-up companies and entrepreneurial ventures at the beginning of their development cycle. This type of investment is crucial for new companies that are looking to grow but do not yet have enough revenue or cash flow to support their operations or expand at the desired pace. Venture capital firms manage pooled funds from many investors to invest in start-ups and small businesses. They typically engage during the early stages of development and play a pivotal role in the growth phases of start-up  companies.
Early-stage investments are considered high-risk because many start-ups fail for various reasons, such as poor market fit and competition. However, the potential for high returns on investment in the event of success is substantial. 
In this paper, we propose and examine a scenario in which the venture capital firm holds a competition among several start-up companies in order to select a top proportion of them at the end of the competition. Specifically, we study behavior of the start-up companies during this competition.

More specifically, we consider a competition between $N$ homogeneous start-up companies over a finite time horizon $\mc{T}=[0,T]$. This competition is set by a venture capital firm that supports the start-ups in this initial phase and aims to select a certain proportion $(1-\alpha)$, $\alpha \in (0,1)$, of top ranking start-ups, based on their market value  (for instance, to receive subsequent funding and support). The proportion  $(1-\alpha)$ is predetermined and announced by the venture capitalist before the competition begins. Subsequently, the market value of the start-up companies is evaluated at the terminal time $T$. 
Those with market values equal to or greater than the sample $\alpha$-quantile of  $N$ companies will be selected by the venture capital firm.

In this application context over the time horizon $\mfT$, $x_t^i$ represents the market value of start-up $i$ at time $t,\, t\in \mfT$. Moreover, $\gamma_t$ represents a deterministic financial, or more generally any form of, support provided by the venture capital firm to each start-up at time $t$. This term captures the practical dynamics of venture capital investment, where funding is typically supplied in stages over time and may include not only monetary contributions but also strategic guidance, managerial expertise, and network access. It is useful to reiterate that  $u^i_t$ represents the effort exerted by start-up $i$ at time $t$ and that the coefficient $b$ is interpreted as the efficiency of this effort. 
The uncertainty specific to the environment of start-up $i$ is modeled by the idiosyncratic Wiener process $w^i$ with volatility $\sigma$. 
Finally, $q^{\alpha, [N]}_T$ denotes the sample $\alpha$-quantile value of the terminal market values of the $N$ start-ups $\{x^i_T\}_{i=1}^N$, as defined in \eqref{general-sample-quantile}, and acts as the threshold for success at the terminal time $T$. 

In this section, we present the results of numerical experiments that provide deeper insights into the venture investment selection criterion under study. \Cref{sec_resuls_threshold} and \Cref{sec_results_target} present the results for the threshold-based formulation (described in \Cref{sec:threshold-general}) and the target-based formulation (described in \Cref{sec:target-general}), respectively, using the model parameters reported in \Cref{model-parameters}\footnote{In the current setup, since $\{\gamma_t\}_{t\in \mathcal{T}}$ is deterministic and identical across all startups, incorporating it into the simulations would merely shift all trajectories by the same amount over time. Consequently, it would not inherently affect the startups' behavior. For this reason, it is set to zero in the numerical experiments. } . Subsequently, \Cref{sec:sensitivity} reports the results of the sensitivity analysis, in which the dynamical and cost functional parameters are varied from their nominal values given in \Cref{model-parameters}. Finally, \Cref{sec:conclusion} discusses that both formulations yield qualitatively similar results across all considered scenarios.
\begin{table}[h]
\centering
\begin{minipage}[c]{0.9\textwidth}
\begin{center}
{\small 
 \begin{tabular}{ccccccc}
\toprule\toprule
 $T$ & $\{ \gamma_t \}_{t\in\mfT}$ & $\mu_0$ & $b$ & $\sigma$ & $r$ &  $\lambda$  \\
 \midrule
 $1$ & $0$& $\mc{N}(0,0.25)$ & $0.5$ &  $0.5$ & $0.1$ & 1\\
\bottomrule\bottomrule
\end{tabular}}
\caption{Model parameters.}
\label{model-parameters}
\end{center}
\end{minipage}
\end{table}

\subsection{Results and Interpretations of Numerical Experiments for the Threshold-Based Formulation}\label{sec_resuls_threshold}
We start with the limiting case where the number of start-ups, $N$, tends to infinity. In this case  the $\alpha$-quantile value, i.e. the threshold for success, becomes deterministic since the distribution of the state is deterministic. The equilibrium $\alpha$-quantile value, $\bar{q}^\alpha$, is the fixed point of the mean-field consistency equation given by \eqref{MF-consistency-general}. We employ the numerical scheme depicted in \Cref{fig:numerical-scheme}, an adaption of the method proposed in \cite{Firoozi2022MAFI}, to completely solve the limiting problem described in \Cref{sec:VC-nonlinear-sol}. 

\Cref{limiting-nonlinear-fixed-alpha} shows the results for the case where only the top $5\%$ of the start-ups at the terminal time will be selected for further funding. This is equivalent to setting the success threshold to the $0.95$-quantile value. Recall that, in the threshold-based formulation, a representative start-up aims to  attain  at least the terminal cutoff threshold. As a result, we observe that the distribution of its market value evolves over time towards the terminal $\alpha$-quantile value, as can be seen in Panel (a) of \Cref{limiting-nonlinear-fixed-alpha}. This fact is further supported by the temporal evolution of the $\alpha$-quantile value depicted in Panel (b) of \Cref{limiting-nonlinear-fixed-alpha}, which illustrates an increasing pattern over time. Another interesting aspect observed in Panel (a) of \Cref{limiting-nonlinear-fixed-alpha} is the increasing concentration of the distribution around the mean value over time, resulting in a decrease in the variance of the distribution. These observations can be attributed to the efforts of a representative start-up company to reach the required threshold by the terminal time, which lead to an increase in its $\alpha$-quantile and a reduction in the dispersion of its market value around the mean.

The optimal strategy for a representative start-up is illustrated in Panel (c) of \Cref{limiting-nonlinear-fixed-alpha}. We observe that, at a fixed point in time, the farther the start-up's market value is from the terminal cutoff threshold, the more effort it exerts. 
Once its market value is ensured to have met this threshold, the representative start-up  ceases to exert additional effort, resulting in the strategy dropping to zero. Furthermore, for a given market value that is below the terminal cutoff threshold, the effort exerted by the start-up increases over time. In other word, as long as its market value has not reached the terminal cutoff threshold, the start-up's effort intensifies with time.     
 \begin{figure}[h] 
 \subfigure[]{
\begin{minipage}[t]{.32 \textwidth} 
    \includegraphics[width=\textwidth]{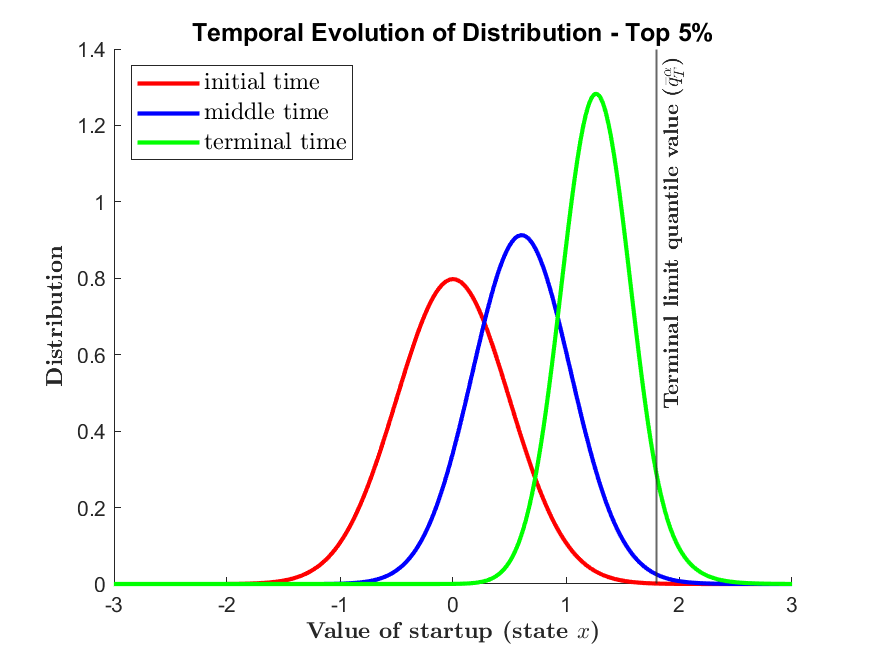}
    \end{minipage}}
\hfill 
\subfigure[]{
\begin{minipage}[t]
{.32\textwidth} 
\includegraphics[width=\textwidth]{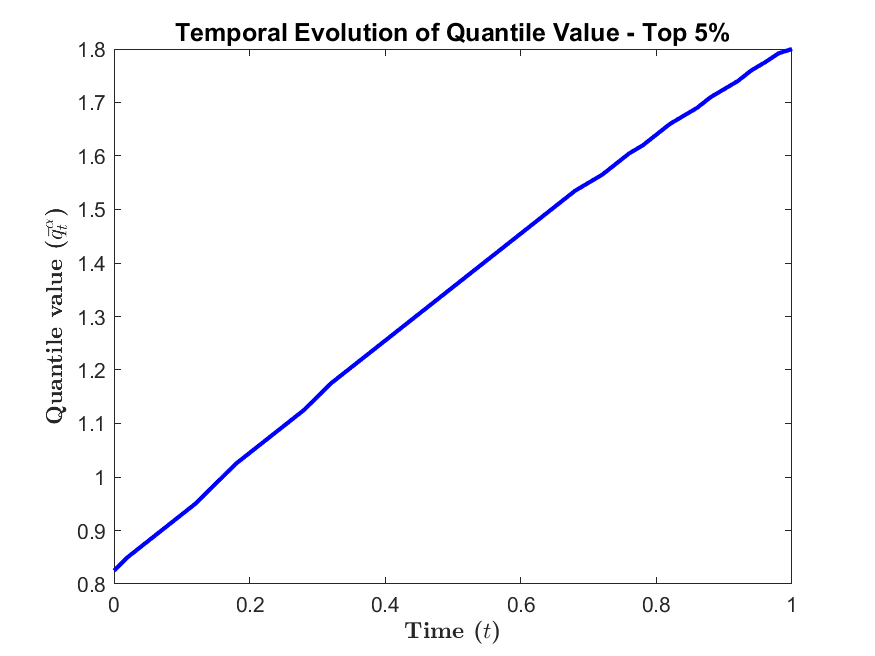}
\end{minipage}}
\hfill 
\subfigure[]{
\begin{minipage}[t]{.32\textwidth} 
 \includegraphics[width=\textwidth]{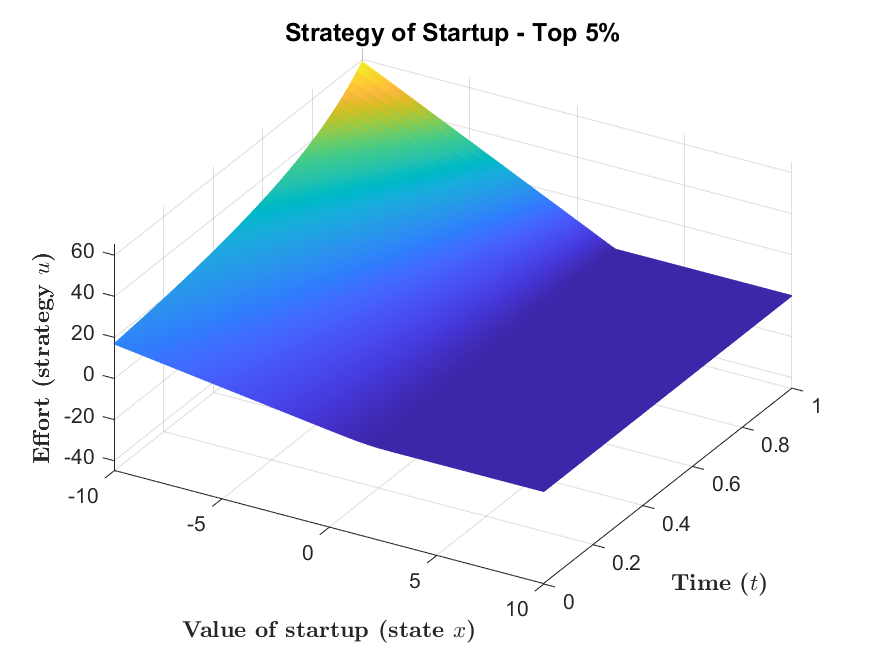}
\end{minipage}}
\caption{Limiting threshold-based model: Results for quantile level $\alpha=0.95$ and parameter values reported in \Cref{model-parameters}. 
}\label{limiting-nonlinear-fixed-alpha}
\end{figure}

Thus far, we have presented numerical results for a specific quantile level, that is $\alpha=0.95$. We now investigate the impact of the quantile level announced by the venture capital firm on behavior of the competing start-ups. \Cref{fig:diff-quantile-levels-nonlinear} illustrates the evolution of $\alpha$-quantile value as a function of the quantile level $\alpha$ and time $t$. 
We observe that the game among participating start-ups becomes more competitive as the quantile level $\alpha$ increases. Specifically, we observe more increase in the quantile value $\bar{q}^\alpha$ through time for larger value of $\alpha$, that is the companies improve their market values further. For relatively low quantile levels (such as $\alpha<0.2$), we observe that the quantile value $\bar{q}^\alpha$ does not increase much, and can even decrease over time. This is while when the quantile level tends to one, the increase in the  quantile value $\bar{q}^\alpha$ over time is the highest. To summarize, the threshold for success at the terminal time, $T=1$, increases with the quantile level $\alpha$. This observation indicates that, for the venture capital firm, setting the bar at a higher level results in a more desirable outcome, that is, the selected start-ups have higher market values.
\begin{figure}
    \centering
    \includegraphics[width=0.4\linewidth]{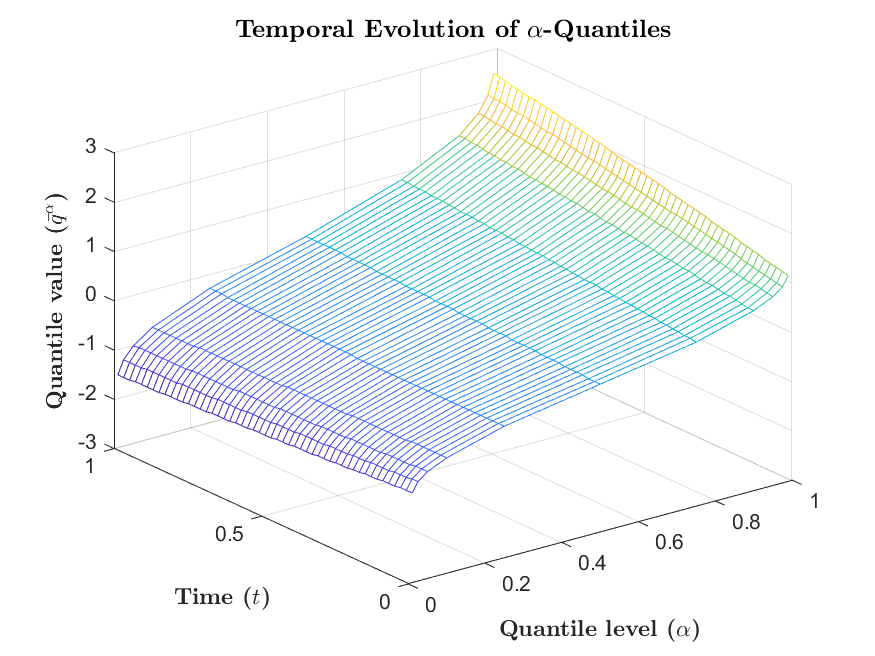}
   \caption{Limiting threshold-based model: Impact of quantile level $\alpha$ on the evolution of equilibrium quantile value $\bar{q}^\alpha$ over time for parameter values reported in \Cref{model-parameters}.} \label{fig:diff-quantile-levels-nonlinear}
\end{figure}

We now apply the optimal strategies  obtained from the limiting case to a competition involving a  finite number of start-up companies, $N=1000$, participating in the competition. To facilitate the comparison of the results with those of the limiting case, we use the same quantile level as before, that is $\alpha=0.95$. 
Panel (a) of \Cref{fig:emp-limit-results-nonlinear} presents the temporal evolution of the limit and empirical distributions
of the start-ups' market values, illustrated by solid lines and histograms, respectively. We observe that the limit distribution approximates the empirical distribution of $1000$ startups over time very well. 
Panel (b) of \Cref{fig:emp-limit-results-nonlinear} depicts the result of one specific simulation, showing both the limit and the empirical quantile values at the terminal time (respectively in pink and red color) 
and the trajectories of $1000$ individual start-ups over time. The start-ups whose trajectories are
depicted in green attain the cutoff threshold at the terminal time, $T=1$, and 
are hence selected by the venture capital firm, while those whose trajectories are depicted in blue are rejected. This specific instance illustrates that that the limit quantile provides a good approximation for the empirical quantile and shows how the competition to reach the terminal threshold leads all the start-up companies to exert efforts and increase their market value. This is due to the cost of not attaining the cutoff threshold, which is increasing with the amount of the shortfall.

Finally, to give a broader insight about how closely the limit quantile, a deterministic quantity, approximates the empirical quantile, a stochastic quantity, Panel (c) of \Cref{fig:emp-limit-results-nonlinear} illustrates a histogram of empirical $\alpha$-quantile values, obtained from $1000$ simulations of the competition involving 1000 start-up companies, along with the limit quantile, depicted by a vertical line. We observe that the empirical quantiles closely match the limit quantile. 

 \begin{figure}[h] 
 \subfigure[]{
\begin{minipage}[t]{.32 \textwidth} 
    \includegraphics[width=\textwidth]{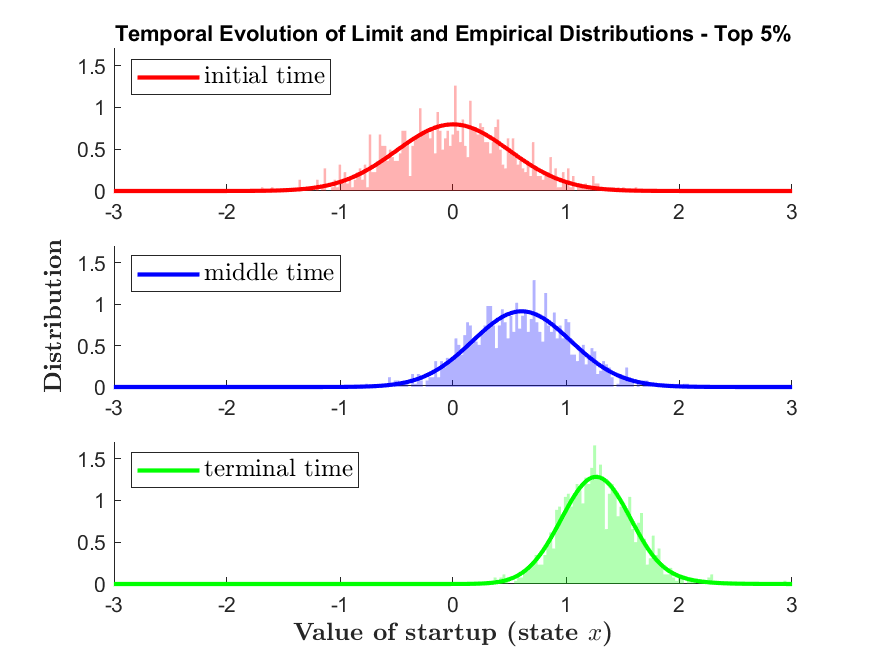}
    \end{minipage}}
\hfill 
\subfigure[]{
\begin{minipage}[t]
{.32\textwidth} 
\includegraphics[width=\textwidth]{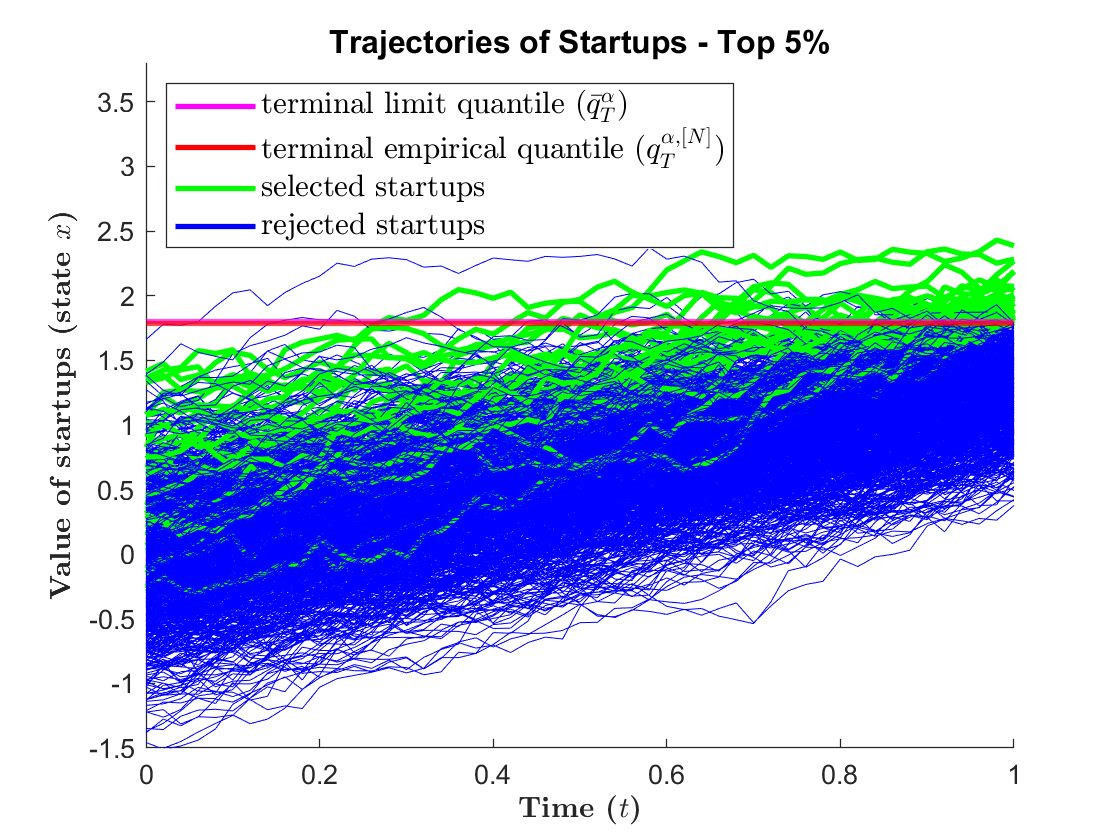}
\end{minipage}}
\hfill 
\subfigure[]{
\begin{minipage}[t]{.32\textwidth} 
 \includegraphics[width=\textwidth]{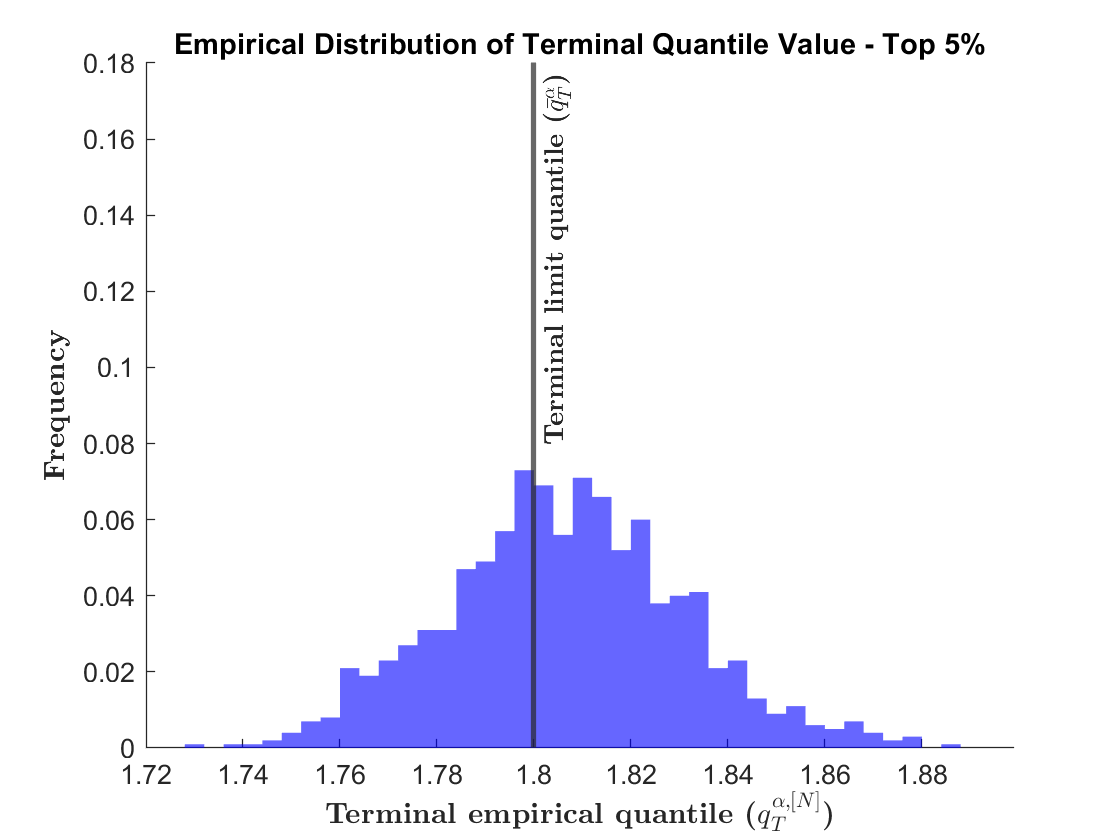}
\end{minipage}}
\caption{Finite-population threshold-based model involving $1000$ start-ups: Results for quantile level $\alpha=0.95$ and parameter values reported in \Cref{model-parameters}.
}\label{fig:emp-limit-results-nonlinear}
\end{figure}

\subsection{Results and Interpretations of Numerical Experiments for the Target-Based Formulation}\label{sec_results_target}
We now consider the target-based formulation, where a representative
start-up company aims to have its terminal market value at exactly a target value, which is defined by the sample $\alpha$-quantile of the participating start-ups. We perform a similar analysis as in Section \ref{sec_resuls_threshold}, starting with the limiting case where the number of start-ups, $N$, tends to infinity.

\Cref{limiting-quadratic-fixed-alpha} shows the results corresponding to a target determined by the $0.95$-quantile of the market-value distribution at the terminal time $T=1$. This figure illustrates  a behavior very similar to the 
one observed for the threshold-based formulation  in \Cref{limiting-nonlinear-fixed-alpha}. 
Specifically, the evolution of the distribution of the representative agent's market value over time, represented in Panel (a) of \Cref{limiting-quadratic-fixed-alpha}, is similar to that illustrated in Panel (a) of \Cref{limiting-nonlinear-fixed-alpha}.
As a result, the temporal evolution of the $\alpha$-quantile value depicted in Panel (b) of \Cref{limiting-quadratic-fixed-alpha} is similar to that illustrated in Panel (b) of  \Cref{limiting-nonlinear-fixed-alpha}. We recall that the $\alpha$-quantile value is analytical for the target-based formulation and is the solution to the ODE given by \eqref{fbode5}. 

The optimal strategy of the representative start-up is depicted in Panel (c) of \Cref{limiting-quadratic-fixed-alpha}, which exhibits a similar trend for positive values of the strategy as seen in Panel (c) of \Cref{limiting-nonlinear-fixed-alpha}. However, in contrast to the threshold-based scenario, where the strategy's value drops to zero for market values exceeding the threshold, the target-based scenario maintains a decreasing strategy as market value increases, aiming to keep it at the target level.

\begin{figure}[h] 
 \subfigure[]{
\begin{minipage}[t]{.32 \textwidth} 
    \includegraphics[width=\textwidth]{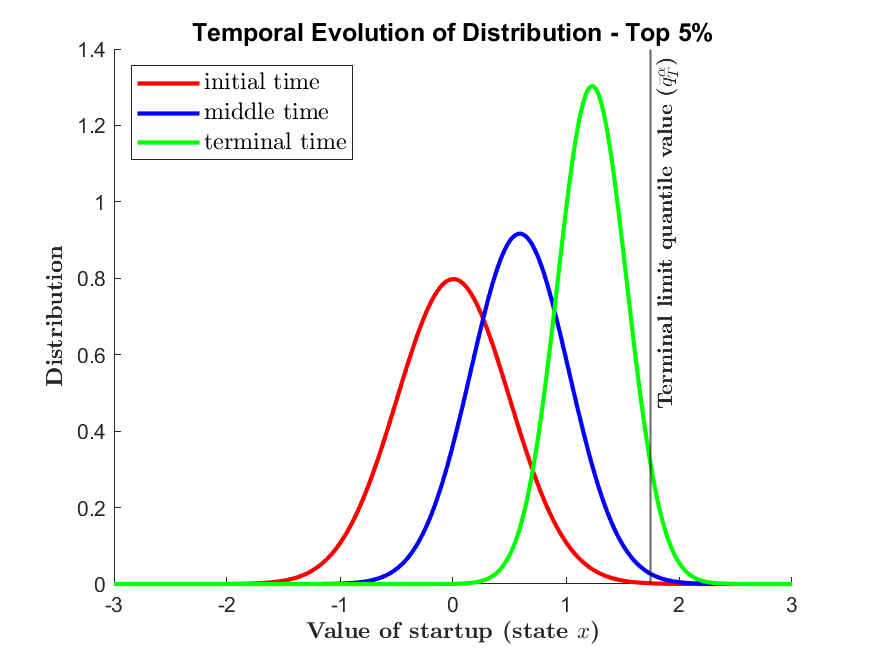}
    \end{minipage}}
\hfill 
\subfigure[]{
\begin{minipage}[t]
{.32\textwidth} 
\includegraphics[width=\textwidth]{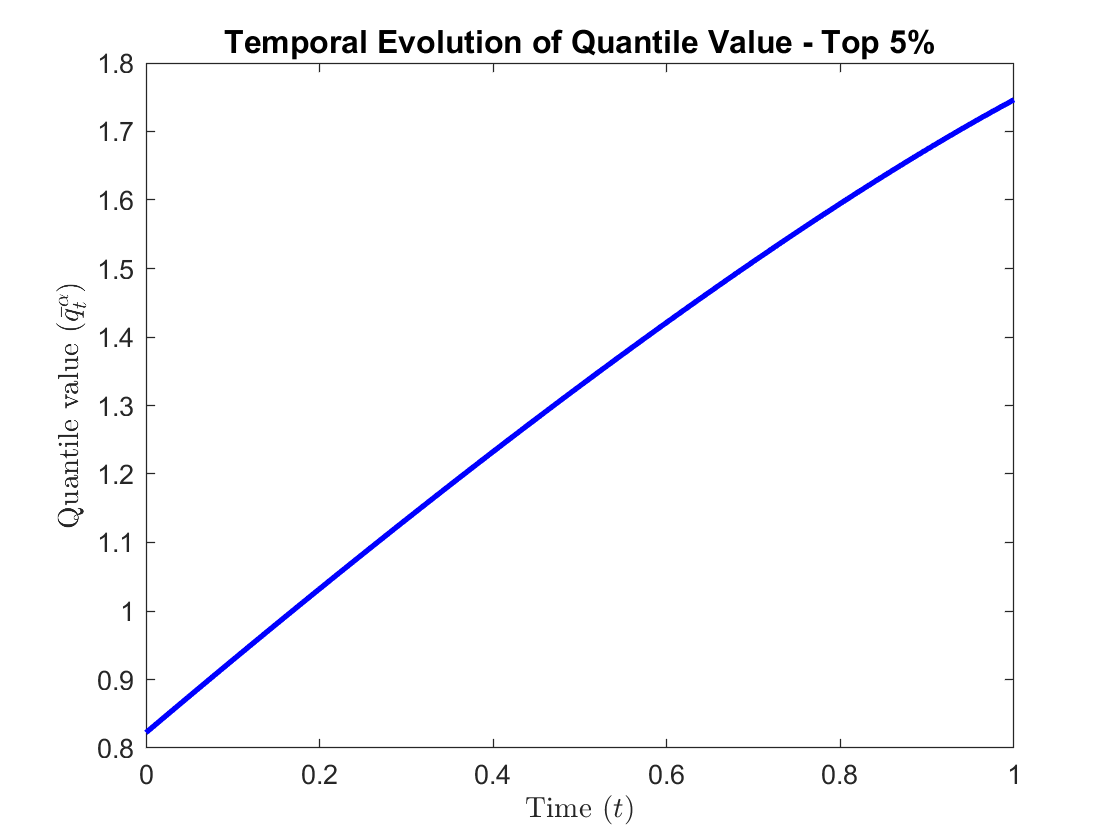}
\end{minipage}}
\hfill 
\subfigure[]{
\begin{minipage}[t]{.32\textwidth} 
 \includegraphics[width=\textwidth]{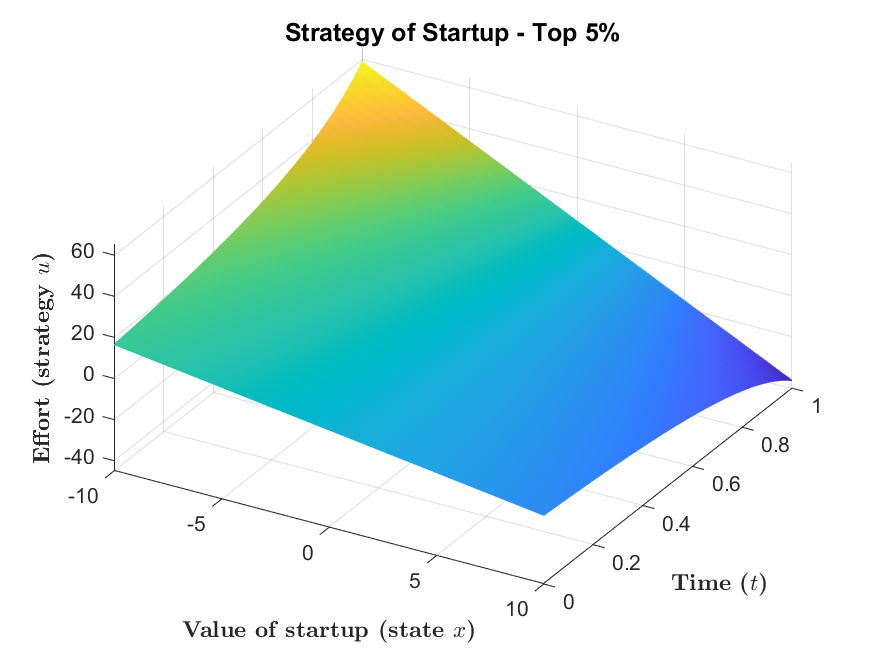}
\end{minipage}}
\caption{Limiting target-based model: Results for quantile level $\alpha=0.95$ and parameter values reported in \Cref{model-parameters}.
}\label{limiting-quadratic-fixed-alpha}
\end{figure}

\Cref{fig:diff-quantile-levels-quadratic} illustrates the impact of the quantile level $\alpha$ on the temporal evolution of the $\alpha$-quantile value. As for the threshold-based formulation (see \Cref{fig:diff-quantile-levels-nonlinear}), an increase in the quantile level $\alpha$ makes the game among start-ups more competitive, leading to a higher terminal target market value.  

\begin{figure}[ht]
    \centering
    \includegraphics[width=0.4\linewidth]{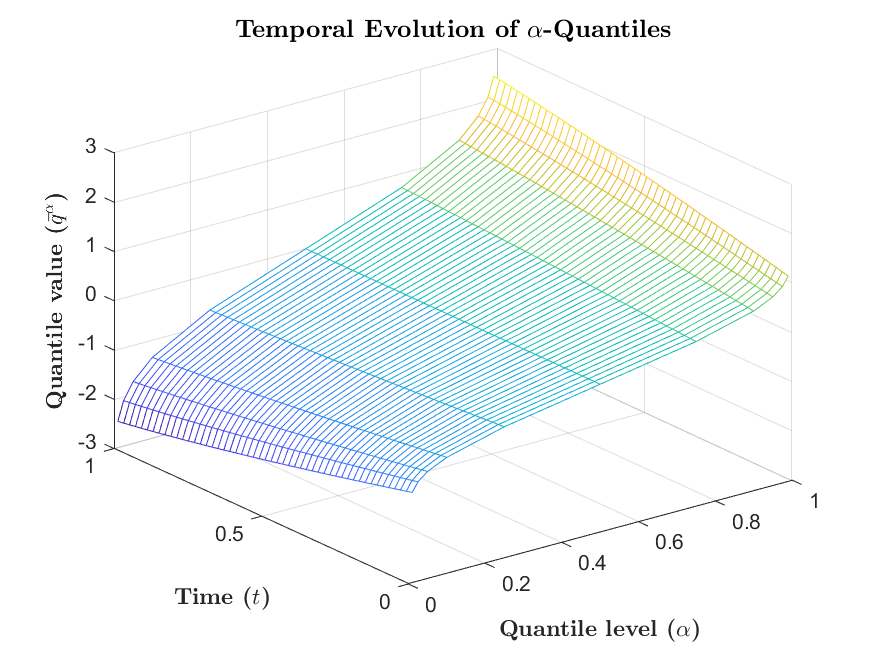}
   \caption{Limiting target-based model: The impact of quantile level $\alpha$ on the evolution of equilibrium quantile value $\bar{q}^\alpha$ over time for parameter values reported in \Cref{model-parameters}.} \label{fig:diff-quantile-levels-quadratic}
\end{figure}

Finally, \Cref{fig:emp-limit-results-quadratic} presents the results corresponding to the competition game among a finite number of start-ups, $N=1000$, with the quantile level set at $\alpha=0.95$. Similarly to the results of the threshold-based formulation, depicted in \Cref{fig:emp-limit-results-nonlinear}, the limit distribution closely approximates the empirical distribution of $1000$ start-ups over time, as illustrated in Panel (a) of \Cref{fig:emp-limit-results-quadratic}. As a consequence, as shown in Panel (b), 
which presents the results of one specific simulation, the limit and the empirical quantile values are almost identical. Panel (b) also illustrates the fact that the terminal market values of the successful start-ups, depicted in green, are slightly more concentrated for the target-based formulation, compared to those of the threshold-based formulation.  
This outcome is expected and is due to the quadratic terminal cost in the target-based formulation, which leads the agents to aim at attaining the target exactly. As a result, the empirical quantile values are slightly more concentrated around the limit quantile in the target-based formulation, compared to those in the threshold-based formulation, as illustrated in Panel (c) of Figures \ref{fig:emp-limit-results-nonlinear} and \ref{fig:emp-limit-results-quadratic}. We expect that the concentration of the sample quantiles around the limit quantile increases with the number $N$ of participating start-ups, which is supported by the Nash error reported in \eqref{Nash-error}. 

 \begin{figure}[h] 
    \subfigure[]{
    \begin{minipage}[t]{.32 \textwidth} 
    \includegraphics[width=\textwidth]{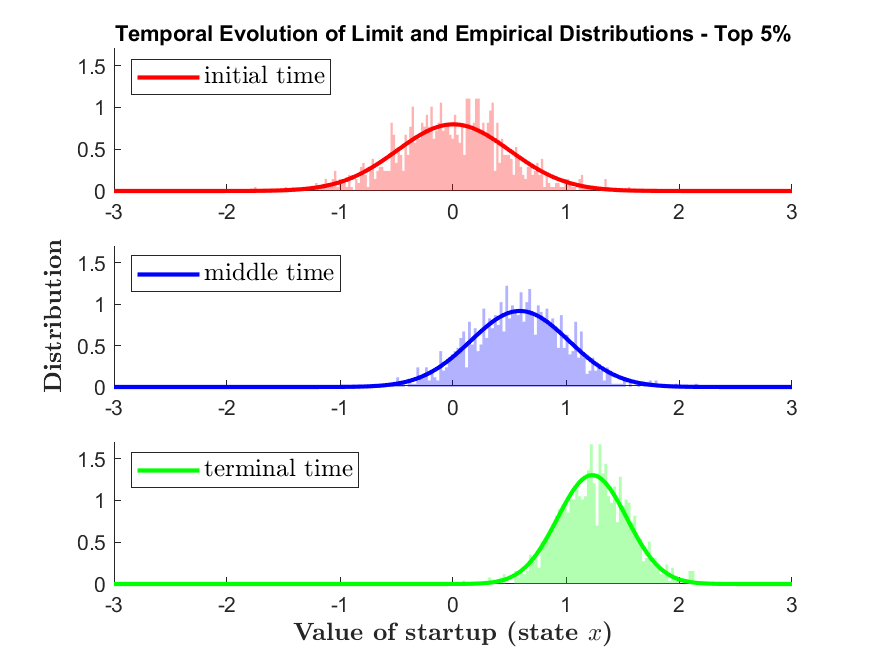}
    \end{minipage}}
\hfill 
    \subfigure[]{
    \begin{minipage}[t]
    {.32\textwidth} 
    \includegraphics[width=\textwidth]{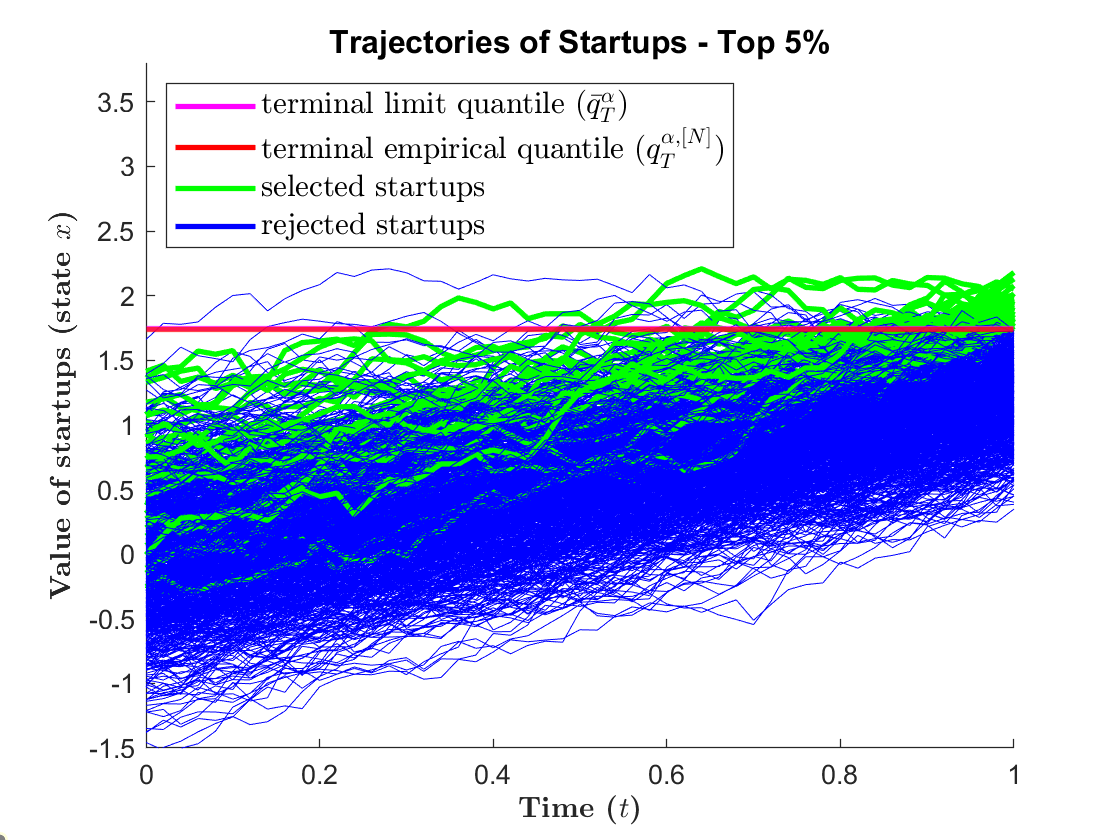}
    \end{minipage}}
\hfill 
\subfigure[]{
\begin{minipage}[t]{.32\textwidth} 
 \includegraphics[width=\textwidth]{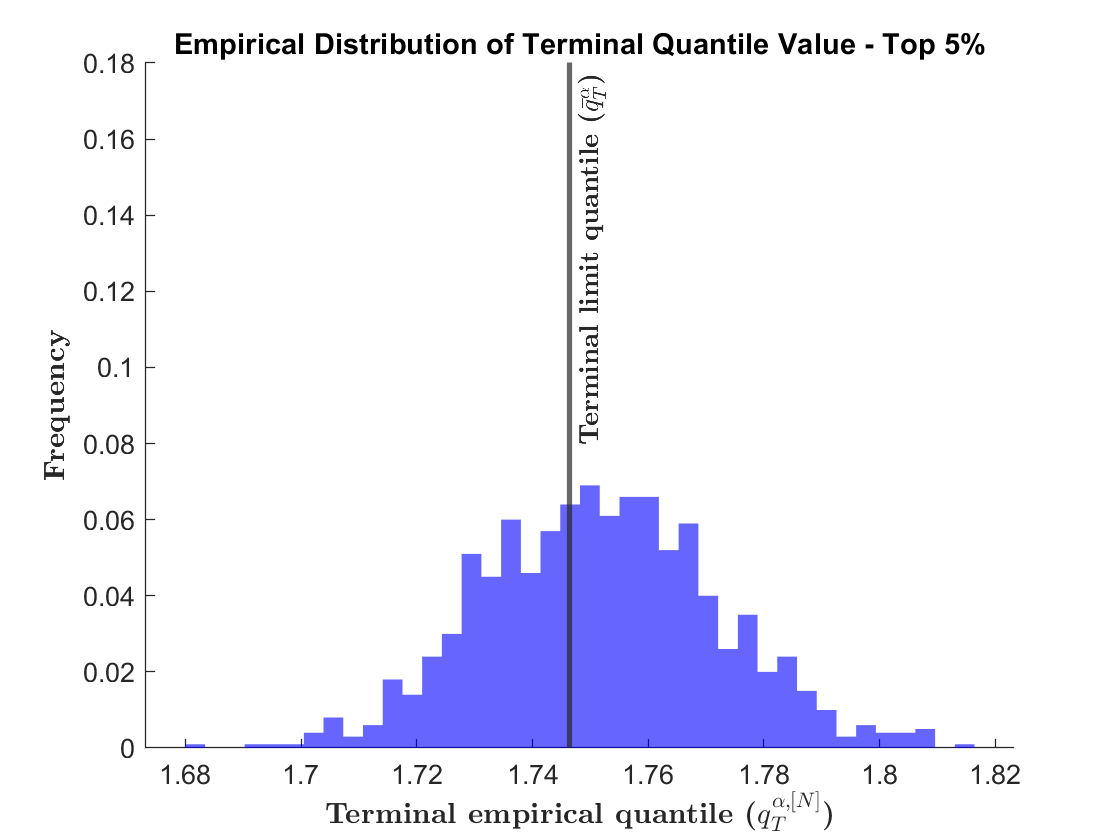}
\end{minipage}}
\caption{Finite-population target-based model involving $1000$ start-ups: Results for quantile level $\alpha=0.95$ and parameter values reported in \Cref{model-parameters}.}
\label{fig:emp-limit-results-quadratic}
\end{figure}

\subsection{Sensitivity Analysis with Respect to Model Parameters}\label{sec:sensitivity}
To evaluate the robustness and reliability of the results, we examine how variations in model parameters affect the equilibrium $\alpha$-quantile value, which is defined as the solution to the quantilized mean-field consistency condition, for both the threshold-based and target-based formulations. Specifically, we vary the dynamical parameters, namely, the efficiency strength $b$ and the volatility $\sigma$, by $\pm 20\%$ and $\pm 40\%$, and present the corresponding results in \Cref{fig:sensitivity-dynamics}. In addition, we vary the cost functional parameters, i.e., the running and terminal cost weights $r$ and $\lambda$, respectively, by the same percentages, and illustrate the results in \Cref{fig:sensitivity-cost}. As shown, in all cases, the equilibrium $\alpha$-quantile for both formulations follows a similar pattern, indicating that the solution obtained from the target-based formulation provides a reasonable approximation to that of the threshold-based formulation across all tested scenarios. Additionally, we observe that the $\alpha$-quantile value for the threshold-based formulation is consistently higher than that of the target-based formulation at equilibrium. This can be explained by the fact that, in the latter, there is a penalty for exceeding the threshold in addition to a cost associated with exerting greater effort. In contrast, in the former, although there remains a cost for increasing effort, there is no penalty for surpassing the threshold. Consequently, the threshold-based formulation is more relaxed, resulting in higher $\alpha$-quantile values at equilibrium.
\begin{figure}[h]
    \centering
    \subfigure[]{
    \begin{minipage}[t]{.38 \textwidth}         \includegraphics[width=\linewidth]{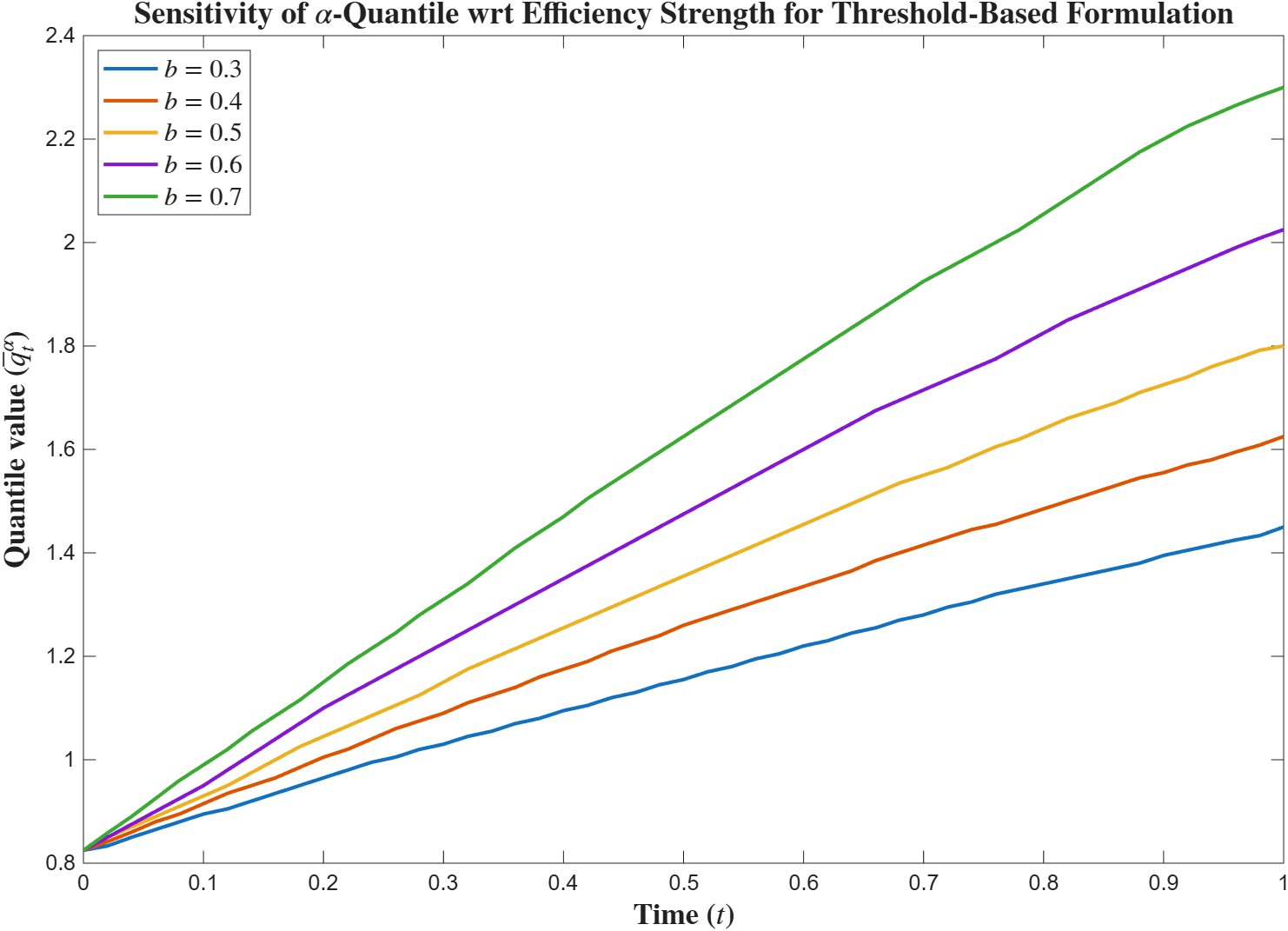}
    \end{minipage}}
    \hspace{0.05\textwidth} 
    \subfigure[]{
    \begin{minipage}[t]{.38 \textwidth} 
        \includegraphics[width=\linewidth]{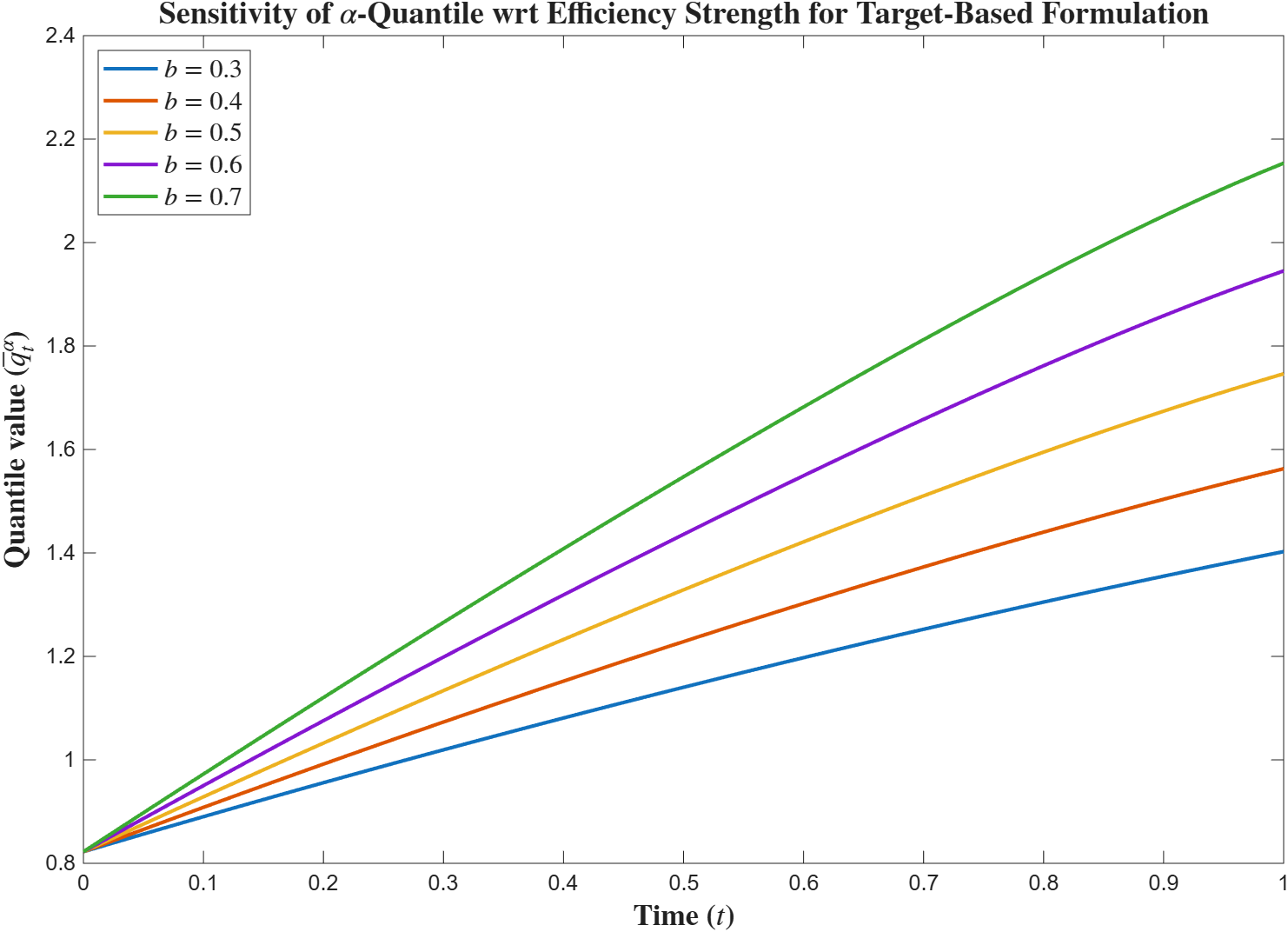}
    \end{minipage}}
    \hfill
    \subfigure[]{
    \begin{minipage}[t]{.38 \textwidth}
        \includegraphics[width=\linewidth]{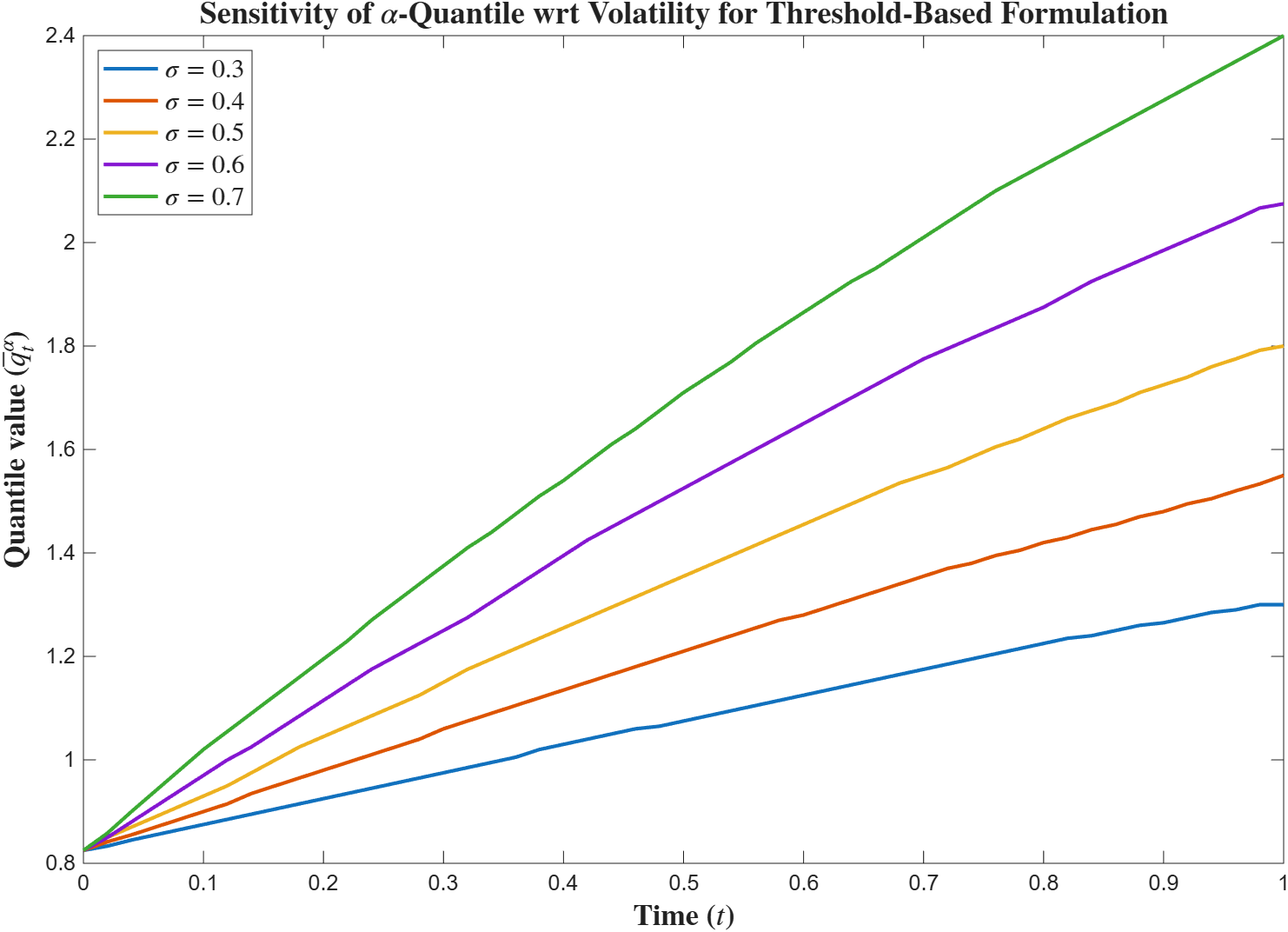}
    \end{minipage}}
    \hspace{0.05\textwidth} 
    \subfigure[]{
    \begin{minipage}[t]{.38 \textwidth}
        \includegraphics[width=\linewidth]{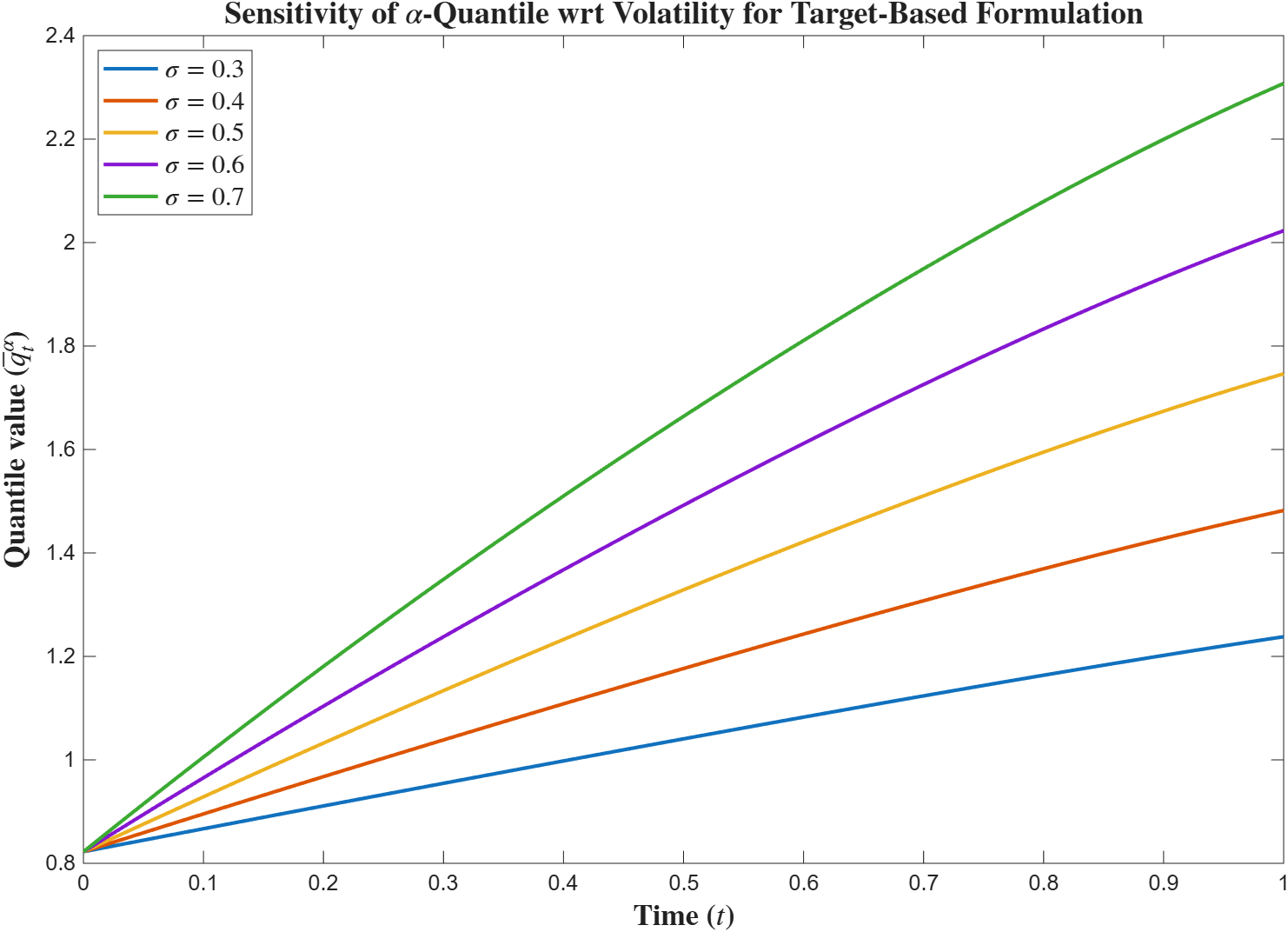}
    \end{minipage}}
    \caption{Sensitivity analysis with respect to the dynamical parameters, efficiency strength $b$ and volatility $\sigma$, varied by $\pm 20\%$ and $\pm 40\%$ from their nominal values $b = 0.5$ and $\sigma = 0.5$.}
    \label{fig:sensitivity-dynamics}
\end{figure}
\begin{figure}[h]
    \centering
    \subfigure[]{
    \begin{minipage}[t]{.38 \textwidth} 
        \includegraphics[width=\linewidth]{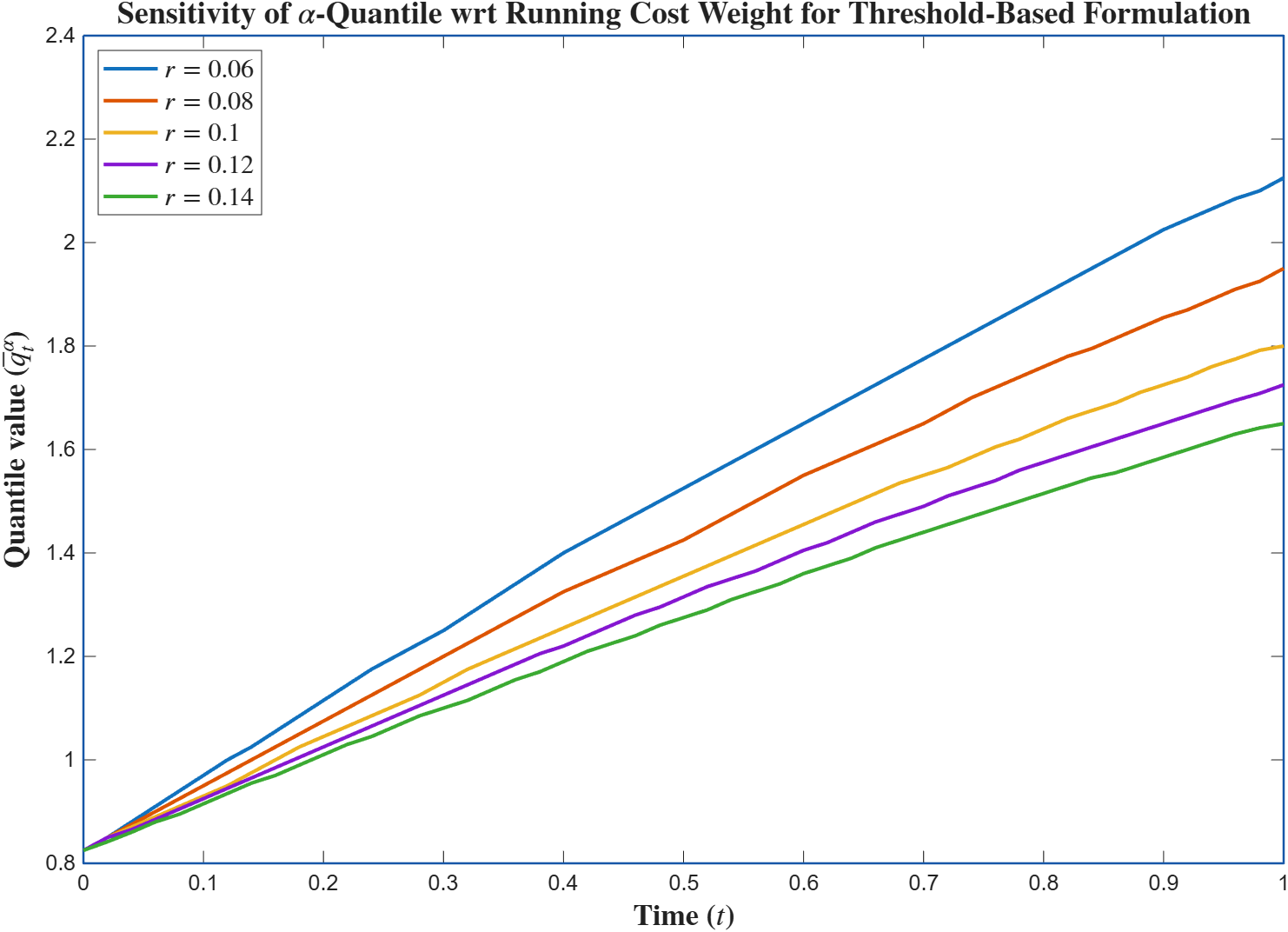}
    \end{minipage}}
        \hspace{0.05\textwidth} 
    \subfigure[]{
    \begin{minipage}[t]{.38 \textwidth}         \includegraphics[width=\linewidth]{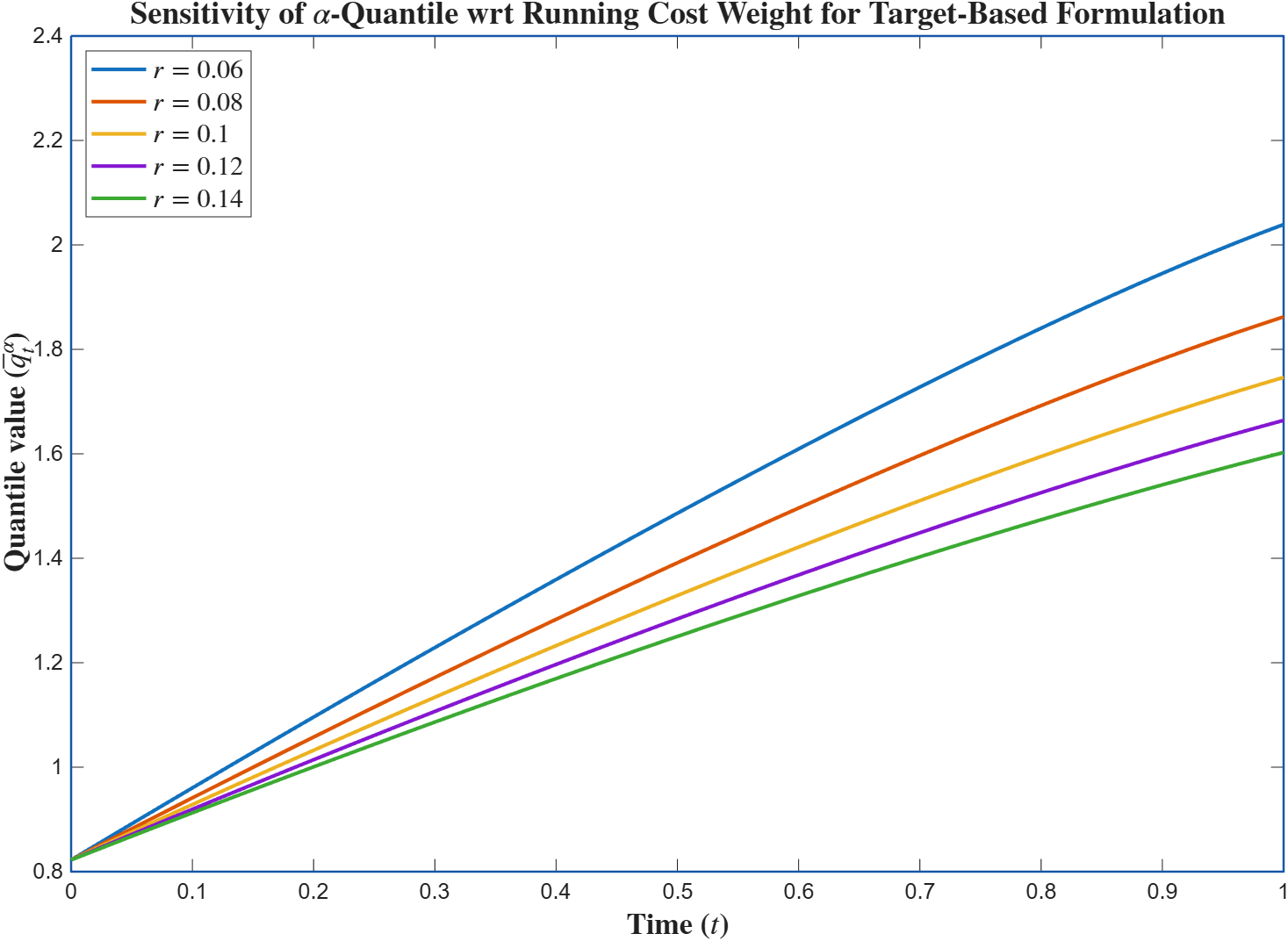}
    \end{minipage}}
\hfill
    \subfigure[]{
    \begin{minipage}[t]{.38 \textwidth}
        \includegraphics[width=\linewidth]{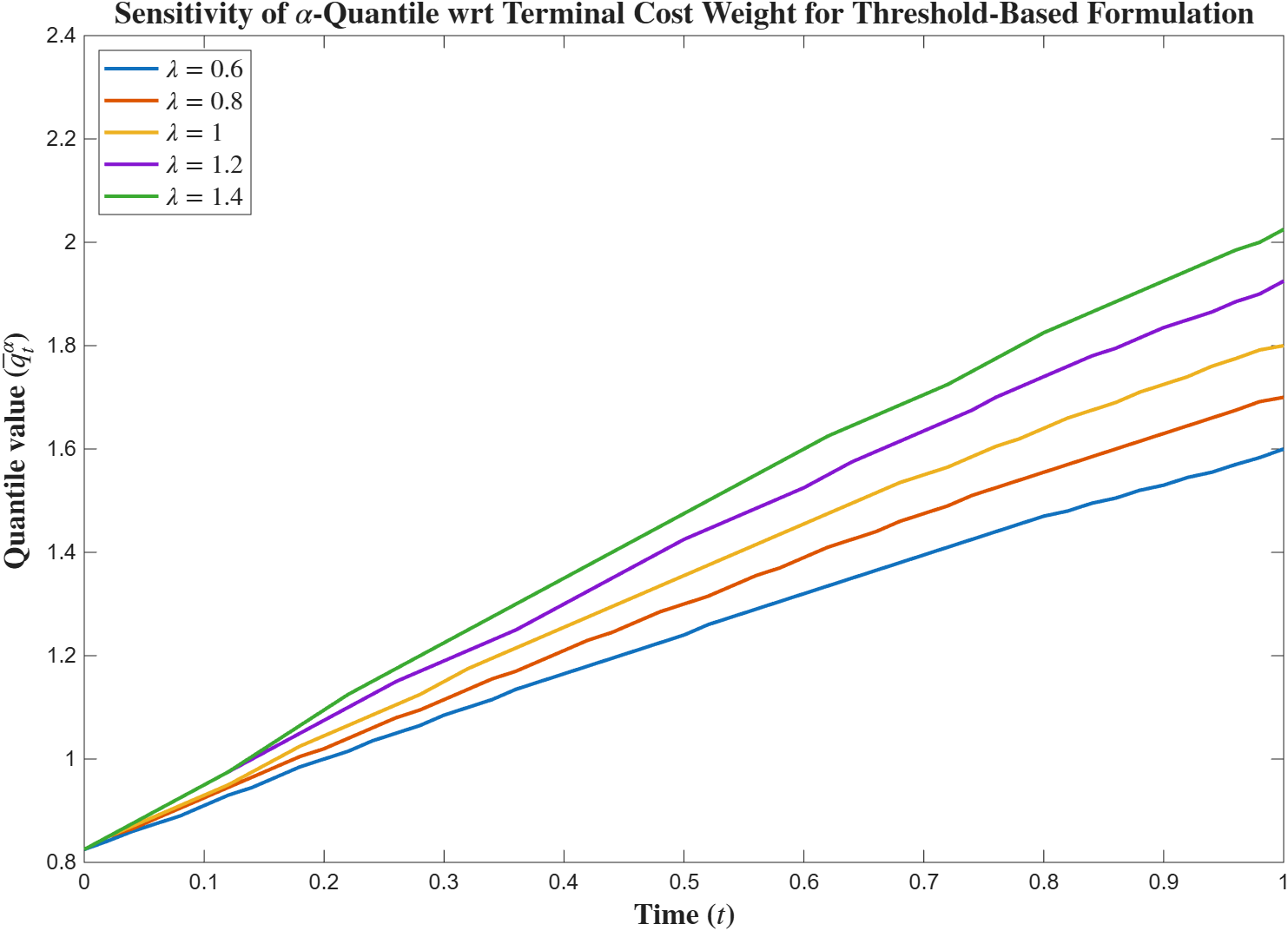}
    \end{minipage}}
    \hspace{0.05\textwidth} 
    \subfigure[]{
    \begin{minipage}[t]{.38 \textwidth}
        \includegraphics[width=\linewidth]{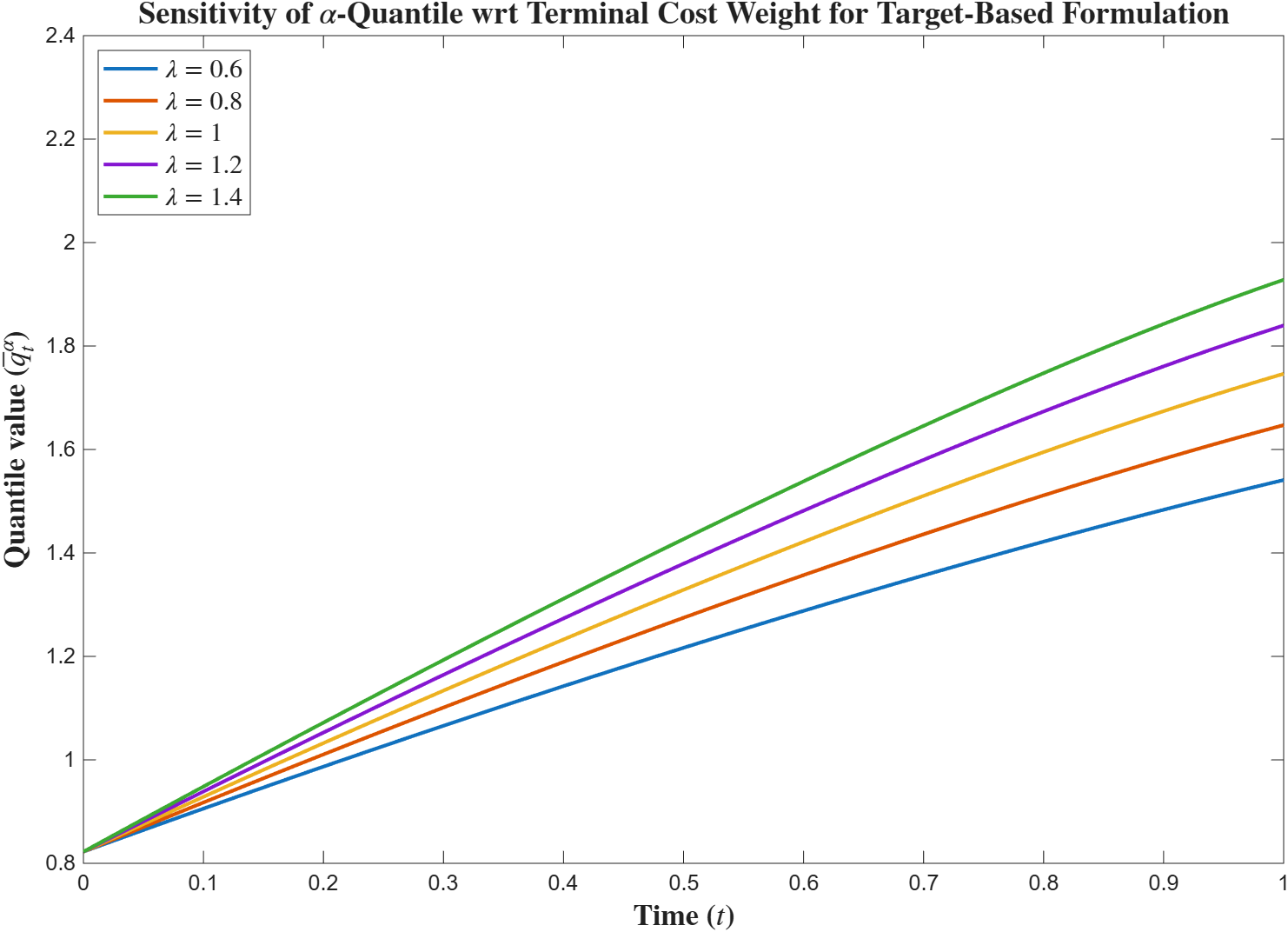}
    \end{minipage}}
    \caption{Sensitivity analysis with respect to the cost functional parameters, running cost weight $r$ and terminal cost weight $\lambda$, varied by $\pm 20\%$ and $\pm 40\%$ from their nominal values $r = 0.1$ and $\lambda = 1$.}
    \label{fig:sensitivity-cost}
\end{figure}
\section{Discussion and Concluding Remarks}\label{sec:conclusion}
For ranking quantilized mean-field games, we introduce the target-based and the threshold-based formulations. The target-based formulation, described by \eqref{general-dynamics-fin-pop}-\eqref{general-sample-quantile} and \eqref{lq-cost-fin-pop}, incorporates a fully quadratic terminal cost. We thoroughly address this problem mathematically and presented results from numerical experiments. The threshold-based formulation, described by \eqref{general-dynamics-fin-pop}-\eqref{general-sample-quantile} and \eqref{cost-fin-pop-VI}, includes a semi-quadratic terminal cost. Although a comprehensive analysis of this problem remains an open mathematical question, we are able to address it numerically. In \Cref{sec:application}, we present the results of numerical experiments within a novel application domain related to early-stage venture investments. Based on these results, we observe that the two formulations lead to very similar $\alpha$-quantile values and distributions of start-up's market values at equilibrium. One may conclude that, in these scenarios, the target-based formulation provides a satisfactory approximation for the threshold-based formulation, which may be more closely aligned with the objectives of competing start-ups. 
This observation can also be explained conceptually. According to \eqref{general-dynamics-fin-pop}-\eqref{general-sample-quantile} and \eqref{cost-fin-pop-VI}, in the threshold-based formulation, start-ups are not incentivized to exert more effort than necessary to fulfill the selection criterion set by the venture capital firm for further funding allocation. This is because as soon as a start-up meets the required threshold its terminal cost drops to zero, while surpassing the threshold implies an increase in the running cost due to the additional effort required. Hence, in this scenario, with surpassing the required threshold, the start-up would incur higher costs to achieve the same outcome, namely qualifying for further funding. 
Consequently, the objective of the start-up in the threshold-based formulation is to ensure that its terminal market value meets the required threshold, without no intention of exceeding it. This objective is close to that of the target-based formulation, where the start-up's goal is to reach a specified terminal market value, with a penalty for exceeding this target. It is noteworthy that in the existing literature, there are instances where quadratic terminal costs are assumed to be a regularized version of put-option-like terminal costs (see, e.g., \cite{Aid-Biagini-2023} in the context of energy markets). 

The advantage of using the target-based formulation to approximate the solution to the threshold-based formulation is that an analytical solution exists for the former, significantly reducing the time required to compute equilibrium strategies. We note that although equilibrium  strategies in the two formulations are different, the endogenous terminal $\alpha$-quantile values and the distributions of the start-ups' market values at equilibrium, are very close in both formulations, which are important for determining the outcome for both start-ups and the venture capital firm. Additionally, from an algorithmic perspective, using the $\alpha$-quantile process obtained from the target-based formulation as the starting point of the numerical scheme presented in \Cref{fig:numerical-scheme} may notably accelerate its convergence.

Finally, the model under study can be extended in several ways to more closely reflect reality in venture investments and other related competitive settings. For example, one could consider heterogeneous agents characterized by different model parameters. In addition, it would be natural to include a common noise term that affects all agents, thereby introducing correlation among them. Although such extensions are feasible, in the interest of simplicity and clarity of exposition in the present paper, they are deferred to future work. Furthermore, another interesting avenue for future research would be to examine the optimization problem faced by the coordinator or venture capital firm, with the aim of characterizing its optimal dynamic funding allocation strategy.

 \bibliographystyle{elsarticle-num} 
 \bibliography{mybib}

\end{document}